\documentclass[11pt,reqno]{amsart}
\usepackage{mathrsfs,graphicx,enumerate}
\usepackage{amsmath,amsfonts,amssymb,amscd,amsthm,bbm}
\usepackage[retainorgcmds]{IEEEtrantools}
\usepackage{enumitem}
\numberwithin{equation}{section}
\usepackage{colortbl}
\usepackage{subcaption}
\usepackage{nccmath}
\usepackage{bm}
\usepackage{hyperref}
\usepackage{cite}
\usepackage{upgreek}
\hypersetup{hidelinks}
\usepackage[pagewise]{lineno}
\allowdisplaybreaks[4]
\title[Emergent behaviors of KMT model]{Emergent behaviors of the kinetic Motsch--Tadmor model in a phase-spatially extended setting}

\author[Ha]{Seung-Yeal Ha}
\address[Seung-Yeal Ha]{\newline Department of Mathematical Sciences and Research Institute of Mathematics \newline Seoul National University, Seoul 08826, Republic of Korea}
\email{syha@snu.ac.kr}

\author[Wang]{Xinyu Wang$^{*}$}
\address[Xinyu Wang]{\newline School of Mathematics, \newline Harbin Institute of Technology, Harbin  150001, People's Republic of China  
	\newline and Research Institute of Mathematics \newline Seoul National University, Seoul 08826, Republic of Korea
}
\email{wangxinyumath@hit.edu.cn}

\begin{document}

\newtheorem{theorem}{Theorem}[section]
\newtheorem{lemma}{Lemma}[section]
\newtheorem{corollary}{Corollary}[section]
\newtheorem{proposition}{Proposition}[section]
\newtheorem{remark}{Remark}[section]
\newtheorem{definition}{Definition}[section]
\newtheorem{example}{Example}[section]
\newcommand{\R}{\mathbb R}
\renewcommand{\arraystretch}{1.5}
	\newcommand*\di{\mathop{}\!\mathrm{d}}

\subjclass[2020]{35Q83,35B40,92D50}
\keywords{kinetic Motsch--Tadmor model, noncompact support, weak flocking, normalized alignment}

\thanks{\textbf{Acknowledgment.} 
	The work of S.-Y. Ha is supported by National Research Foundation (NRF) grant funded by the Korea government (MSIT) (RS-2025-00514472). The work of X. Wang is supported by the Natural Science Foundation of China (grants 12601001, 123B2003), the China Postdoctoral Science Foundation (grants 2025M774290), Heilongjiang Province Postdoctoral Funding (grants LBH-Z24167), and the Fundamental Research Funds for the Central Universities. $^{*}$ Corresponding author.}
\begin{abstract}	
The kinetic Motsch--Tadmor (in short, KMT) model is a kinetic flocking model with a normalized communication weight. In this paper, we study the emergent dynamics of the phase-spatially extended KMT model. We first establish a global well-posedness theory in the fully noncompact spatial--velocity setting. To this end, we introduce a direct Lagrangian formulation in which the unbounded part of the initial velocity distribution is separated from an interaction-generated bounded remainder. This decomposition allows us to construct global Lagrangian weak solutions without truncating the velocity distribution and to propagate finite phase-space moments. We then investigate the long-time collective behavior of the resulting solutions. When the initial velocity support is compact while the spatial support is allowed to be noncompact, a time-varying effective-region argument yields a uniform contraction mechanism for the normalized interaction and leads to exponential weak support flocking. When both the spatial and velocity supports are noncompact, support-level flocking is in general impossible. Nevertheless, the same effective-region mechanism, combined with the Lagrangian decomposition and a bootstrap argument for the interaction-generated remainder, yields exponential weak moment flocking. In particular, pairwise spatial moments remain uniformly controlled while velocity fluctuations converge exponentially to zero. These results provide a unified framework for the well-posedness and flocking dynamics of the KMT model beyond the compact-support regime.
\end{abstract}

\maketitle


\section{Introduction} \label{sec:1}
\setcounter{equation}{0}
The term ``{\it flocking}'' denotes a collective phenomenon in which self-propelled particles organize into coherent motion using limited information about their environment and neighboring agents. An early computational formulation was introduced by Reynolds \cite{Re}, followed by the influential works of Vicsek et al. \cite{V-C-B} and Cucker--Smale (in short, CS) \cite{Cucker,cucker2}. Since then, several mathematical models have been proposed in the study of collective flocking behaviors. When the number of agents is sufficiently large, corresponding kinetic models provide an effective framework for describing the collective dynamics of large populations (see \cite{flocking1,flocking2,flocking3,flocking4,CFRT2010,J2,A1,w2,w5}). In this paper, we focus on the phase-spatially extended kinetic Motsch–Tadmor (in short, KMT) model introduced in \cite{MT}. Due to its non-symmetric alignment mechanism and rich flocking dynamics, the MT model has attracted considerable attention from various perspectives \cite{flocking1,W-MT-2026,flocking2}. Nevertheless, most existing studies are restricted to the case where the particle distribution has compact support, which significantly limits the scope of the analysis. In this work, we investigate the emergent dynamics of the spatially extended KMT model without imposing a compact-support assumption.

To set up the stage, we begin with a brief introduction to the MT model \cite{MT}. Let $x_i = x_i(t)$ and $v_i = v_i(t)$ be the position and velocity of the $i$-th particle at time $t$. Then, the MT model is a Newton-like model with a forcing term given by the normalized weighted sum of relative velocities, and its Cauchy problem reads as follows. 
\begin{equation}
\begin{cases} \label{A-1}
\displaystyle {\dot x}_i=v_i, \quad t > 0, \quad i \in [N]:=\left\{1,2,...,N\right\}, \vspace{8pt}\\
\displaystyle {\dot v}_i =\kappa \sum_{j\in [N]}\phi_{ij} (v_j - v_i),\quad \kappa>0, \vspace{8pt}\\
\displaystyle (x_i, v_i) \Big|_{t = 0} = (x_{i0},~v_{i0}),
\end{cases}
\end{equation}
where $\phi$ and $\phi_{ij}$ are the CS communication weight and its normalized counterpart, respectively, given by
\begin{equation}  \label{A-2}
\phi(r) :=\frac{1}{(1+r)^{\beta}},~~ r \geq 0,~~\beta \geq 0 \quad  \mbox{and} \quad \phi_{ij} := \frac{\phi(|x_j - x_i|)}{\sum_{k \in [N]} \phi(|x_{k} - x_i|)}, \quad \forall~i, j \in [N].
\end{equation}
Here, $| \cdot |$ is the standard $\ell^2$-norm in ${\mathbb R}^d$. \newline

The MT model \eqref{A-1}, introduced in \cite{MT}, describes flocking in interacting particle or multi-agent systems and may be viewed as a normalized variant of the CS model. Motsch and Tadmor showed that the model can perform better than the CS model \cite{Cucker,cucker2,k12,w1,w3,w4,M-P,DS1,DS2} in a far-from-flocking regime. It has since been extensively investigated in the literature from various perspectives, e.g., short-range flocking \cite{J2, flocking2}, hydrodynamic limit and kinetic model \cite{flocking1, A1}, multi-cluster flocking \cite{HJZ2025}, spectral analysis \cite{HWYMT}, infinite particle system \cite{W-MT-2026,HWXMT}, and so on. In this paper, we investigate the emergent dynamics of the Cauchy problem for the KMT model  with noncompact phase-spatial support as follows.
\begin{equation} 
	\begin{cases} \label{A-3}
			\displaystyle	\partial_t\mu_t+v\cdot\nabla_x\mu_t+\nabla_v\cdot\bigl(L[\mu_t]\mu_t\bigr)=0, \quad t > 0,~~(x, v) \in {\mathbb R}^{2d},  \vspace{8pt}\\
			\displaystyle 	L[\mu](x,v) :=\kappa\bigl(\bar v_\mu(x)-v\bigr),
			\quad
			\bar v_\mu(x)
			:=
			\frac{\displaystyle\int_{\R^{2d}}\phi(|x-x_{\star}|)v_{\star}\,\mu(\di x_{\star},\di v_{\star})}
			{\displaystyle\int_{\R^{2d}}\phi(|x-x_{\star}|)\,\mu(\di x_{\star},\di v_{\star})}, \vspace{8pt}\\
			\displaystyle \mu \Big|_{t = 0}  =\mu_0.
		\end{cases}
\end{equation} 

\noindent In real-world systems, noncompact scenarios play a critical role in the emergence of cooperative behavior among agents, particularly in the presence of stochastic environmental forcing \cite{Shvydkoy2021,Shvydkoy2024,Shvydkoy2026,ks2}. Many familiar distributions, including Gaussian and Cauchy distributions, have noncompact support, and finite-moment distributions naturally arise in this setting. On the other hand, the noncompact distribution can reflect weak graph network interaction at infinity. In \cite{HaWangXue2025,H-W-CS-2026,HWGKJK2026,H-W-2025}, the authors utilized the infinite graph mean-field limit, average velocity conservation, and kinetic energy dissipation to show that weak flocking behavior emerges for the symmetric kinetic CS-type model with noncompactly supported initial data. The weak flocking behavior in references \cite{HaWangXue2025,H-W-CS-2026,HWGKJK2026,H-W-2025} means that the second moment for the velocity deviation from an initial average velocity tends to zero asymptotically, while the second moment for spatial deviation from the center of mass remains bounded uniformly in time. However, as a variant of the CS model without average velocity conservation and dissipative structure of kinetic energy, the emergent behavior of the KMT model \eqref{A-3} with noncompact spatial-velocity support has yet to be explored. This motivates us to consider the following questions:
\vspace{0.2cm}

\begin{itemize}
	\item \noindent  (Q1):~If the initial data are not compactly supported, does there exist a corresponding well-posedness theory to the solution to the KMT model \eqref{A-3}?
	\vspace{0.2cm}
	\item \noindent  (Q2):~If so, does the solution still show some corresponding emergent behavior under noncompact support?
\end{itemize}
\vspace{0.1cm}

\noindent For (Q1), the global well-posedness to \eqref{A-3} with compact support has been already studied by Gao and Xue \cite{GaoXue} in the class of measure-valued solutions using the particle approximation and the Monge--Kantorovich--Rubinstein distance.  Moreover, the weak solution of the KMT model with special cut-off communication function has also been studied in \cite{kinetic3}. In the fully noncompact finite-moment regime considered here, we include a direct Lagrangian construction tailored to the decomposition. More precisely, let $z=(x,v)$ denote the initial Lagrangian label and we set
\[
m_0:=\int_{\R^{2d}}v_{\star}\,\mu_0(\di x_{\star},\di v_{\star}), \quad q_0(z):=v-m_0.
\]
Let
$
(X(t,z),V(t,z))
$
denote the particle trajectory associated with the kinetic equation, where \(X(t,z)\) and \(V(t,z)\) are its position and velocity components,  respectively. More precisely, the forward particle trajectories issued from $z$ at time $t= 0$ solve
\begin{equation}\label{A-4}
	\begin{cases}
		\displaystyle\dot X(t,z)=V(t,z), \quad t > 0,~~z \in {\mathbb R}^{2d}, \vspace{8pt}\\
		\displaystyle \dot V(t,z)
		=L[\mu_t]\bigl(X(t,z),V(t,z)\bigr),\vspace{8pt}\\
		X(0,z)=x,\quad V(0,z)=v.
	\end{cases}
\end{equation}
Under some suitable assumptions on the communication kernel $\phi$ and the solution \(\mu\), the map $z~\longmapsto
\bigl(v,L[\mu_t](z)\bigr)$ is locally Lipschitz in the phase variables. Hence the characteristic system admits a unique local solution by the classical ODE theory. Moreover, since the acceleration has at most linear growth in \(V\), the usual Grönwall estimate prevents finite-time blow-up, and the characteristic trajectory extends to the whole time interval on which \(\mu\) is defined.
Solving the second equation in \eqref{A-4} by the variation of constants, we have
\begin{align}
\begin{aligned} \label{A-5}
V(t,z)
&=
m_0+e^{-\kappa t} (v-m_0)
+\kappa\int_0^t e^{-\kappa(t-s)}
\Bigl (\bar v_{\mu_s}\bigl(X(s,z)\bigr)-m_0\Bigr)\di s\vspace{6pt} \\
&:= \underbrace{m_0}_{\text{fixed reference}}
	+
	\underbrace{e^{-\kappa t}q_0(z)}_{\text{damped initial tail}}
	+
	\underbrace{U(t,z)}_{\text{interaction remainder}}.
\end{aligned}
\end{align}
Note that the remainder term $U$ is not the fluctuation around the current mean velocity which is not a conserved quantity.  Rather, $U$ (see Lemma \ref{L4.1}) is the interaction-generated correction after the pointwise-unbounded initial tail has been separated into the explicit factor $e^{-\kappa t}q_0$. This makes the existence, uniqueness, and flocking estimates part of a single analytic scheme. For more details, we refer to Proposition~\ref{P3.1} and Section \ref{sec:4}. \newline 

For (Q2), the appropriate notion of emergent behavior depends essentially on whether the velocity support is bounded.  To make this distinction precise, for $p \in [1, \infty)$, we set
\begin{align}
\begin{aligned}  \label{A-6}
	\mathcal X_p(t)
	&:=
	\left(
	\iint_{\R^{4d}}
	|X(t,z)-X(t, z_{\star})|^p
	\mu_0(\di z)\mu_0(\di z_{\star})
	\right)^{1/p}, \\
	\mathcal V_p(t)
	&:=
	\left(
	\iint_{\R^{4d}}
	|V(t,z)-V(t, z_{\star})|^p
	\mu_0(\di z)\mu_0(\di z_{\star})
	\right)^{1/p}.
\end{aligned}
\end{align}
These functionals denote two-point spatial and velocity correlation functions in $L^p$-setting.  When the velocity support is bounded, we also use the velocity-support diameter functional:
\begin{equation}\label{A-7}
	\mathcal D_v(t)
	:=
	\operatorname*{ess \ sup}_{z, z_{\star} \in {\mathbb R}^{2d}}
	|V(t,z)-V(t, z_{\star})|=\mathcal  V_{\infty}(t).
\end{equation}
\begin{definition}[Weak flocking]\label{D1.1}
Let $\mu_t$ be a global Lagrangian weak solution to the KMT model \eqref{A-3}. 
	\begin{enumerate}
		\item $\mu$ exhibits weak support flocking of order \(p\) if the following estimates hold:
		\[		\sup_{t\ge0}\mathcal X_p(t)<\infty,
		\qquad
		\lim_{t\to\infty}\mathcal D_v(t)=0.
		\]
		\item $\mu$ exhibits weak moment flocking of order \(p\) if the following estimates hold:
		\[
		\sup_{t\ge0}\mathcal X_p(t)<\infty,
		\qquad
		\lim_{t\to\infty}\mathcal V_p(t)=0.
		\]
	\end{enumerate}
\end{definition}
We use the communication weight function \eqref{A-2} as a particular example to show our two weak flocking results. That is, we set $\beta=\alpha\in[0,1]$ as in Definition \ref{D3.1} and Remark \ref{R3.1}. We first consider initial data with compact velocity support but possibly noncompact spatial support.  In this regime, velocity diameter \(\mathcal D_v(t)\) is finite and  it is a velocity observable.  Rather than enclosing the entire spatial distribution in a bounded region, we select at each time a moving effective region carrying at least one half of the total mass (see Lemma \ref{L5.1}).  The normalization in \(\eqref{A-3}_2\), together with the algebraic moderateness of the communication function, implies that every normalized interaction kernel assigns a common positive amount of mass to this region (see Lemma \ref{L5.2}).  The corresponding normalized averaging operator therefore satisfies a Dobrushin-type contraction estimate for bounded velocity observables (see Lemma \ref{L5.3} and Lemma \ref{L5.4}).  More precisely, the velocity-support diameter satisfies an inequality of the form
\[
\begin{cases}
\displaystyle D^+\mathcal X_p(t)\le \mathcal D_v(t), \quad \mbox{a.e.}~t > 0, \vspace{6pt}\\
\displaystyle D^+\mathcal D_v(t)
\le
- \frac{\kappa}{2C_{\phi}}
\min\bigl\{6^{-\alpha},d_{\phi}\bigr\}
\bigl(1+\mathcal X_p(t)\bigr)^{-\alpha}
\mathcal D_v(t), \quad  d_\phi>0, 
\end{cases}
\]
where $D^+$ is the upper Dini derivative, $C_\phi, d_\phi$ are positive structural constants associated with $\phi$ (see Definition \ref{D3.1})  and $\alpha \in [0,1].$ Then, the nonlinear Lyapunov functional in \cite{flocking4} yields the following weak support flocking of order $p$:
\[
\sup_{t\ge0}\mathcal X_p(t)<\infty,
\quad
\mathcal D_v(t)\le Ce^{-\lambda t}, \quad t \geq 0.
\]
For more details, we refer to Theorem \ref{T3.1} and Section \ref{sec:5}.

To treat fully noncompact spatial-velocity support regime directly, we again use the variation-of-constants formula for the velocity characteristic  decomposition \eqref{A-5}. Two analytic properties of \eqref{A-5} are crucial.  First, although \(q_0\) does not belong to \(L^\infty(\mu_0)\), its normalized average \(\mathscr T_{X(t)}q_0\) (see \eqref{D-2} for the definition of $\mathscr T_{X(t)}$) is uniformly bounded in the observation variable, whenever the required \(p\)-moments are finite (see Lemma~\ref{L6.2}).  Thus the normalized MT interaction converts the fixed unbounded velocity tail into a bounded forcing term.  Second, the same time-varying effective-region argument gives a common-mass contraction for the oscillation of \(U\). More precisely, we set the diameter of $U$ and auxiliary spatial scale as follows.
$$
\mathcal{D}_U(t):=
\operatorname*{ess\,sup}_{z,z_{\star}\in{\mathbb R}^{2d}}
|U(t,z)-U(t, z_{\star})|, \quad 
\mathcal{S}_U(t):=
1+\mathcal X_p(0)
+\frac{1}{\kappa}\mathcal V_p(0)
+\int_0^t\mathcal D_U(s)\,\di s.
$$
Then, it follows from Lemma \ref{L6.3} and Lemma \ref{L6.4} that there exist some positive constants $a,b>0$ such that
\begin{equation}\label{A-8}
\begin{cases}
\displaystyle  1+\mathcal X_p(t)\le \mathcal{S}_U(t), \quad t \geq 0, \vspace{6pt}\\
\displaystyle 
D^+\mathcal{D}_U(t)
	\le
	-a\mathcal{S}_U(t)^{-\alpha}\mathcal{D}_U(t)
	+
	be^{-\kappa t}\mathcal{S}_U(t)^\alpha, \vspace{6pt} \\
\displaystyle \mathcal{S}_U'(t)=\mathcal{D}_U(t),  \quad \mbox{a.e.,}~t > 0, \vspace{6pt}\\
\displaystyle ({\mathcal D}_U, \mathcal{S}_U) \Big|_{t = 0+} = \Big (0,~1+\mathcal X_p(0)
+\frac{1}{\kappa}\mathcal V_p(0) \Big).
\end{cases}
\end{equation}
Then, this set of estimates in \eqref{A-8} implies exponential weak moment flocking of order $p$:
$$
\sup_{t\ge0}\mathcal X_p(t)<\infty,
\quad
\mathcal V_p(t)\le Ce^{-\lambda t}.
$$  
For more details, we refer to Theorem \ref{T3.2} and Section \ref{sec:6}.\newline 

A main novelty of the present work is that the time-varying effective-region method is developed without momentum conservation. In the previous symmetric CS-type studies \cite{H-W-2025,H-W-2026,H-W-CS-2026,HWGKJK2026,k6}, the conservation of total momentum was crucially used in the flocking analysis, i.e., conservation of the average velocity determines the center-of-mass motion explicitly, so the effective region can be localized around a fixed reference state, whereas the MT normalization destroys momentum conservation, and no such reference frame is available a priori. Moreover, in the compact-velocity regime of the previous works, the velocity diameter may remain constant in time \cite{HWGKJK2026,HWY}, while in the present MT model, the velocity-support diameter itself contracts exponentially. Thus, the normalized interaction not only allows us to recover weak flocking in the absence of momentum conservation but also yields a stronger dynamical contraction mechanism away from equilibrium. This leads to exponential weak support flocking for compact velocity support and exponential weak moment flocking in the fully noncompact regime.\newline

The rest of this paper is organized as follows. In Section~\ref{sec:2}, we recall the concepts of weak solution and characteristic flow, and then we summarize previous results. In Section \ref{sec:3}, we describe the framework and summarize our main results. In Section~\ref{sec:4}, we provide a global Lagrangian well-posedness in the fully noncompact finite-moment regime. In Section~\ref{sec:5}, we show the exponential weak support flocking for compact velocity support. In Section~\ref{sec:6}, we establish exponential weak moment flocking in the fully noncompact regime. Finally, Section~\ref{sec:7} is devoted to a brief summary of the main results and some remaining issues for future work. In Appendix \ref{app-A}, we provide the proof of Lemma \ref{L6.5}.

\vspace{0.2cm}

 \section{Preliminaries}\label{sec:2}
 \setcounter{equation}{0}
In this section, we recall basic notions for weak solution and characteristic flow to be used in later analysis, and recall previous results on the emergent dynamics for the KMT model in a phase-spatially confined setting.

\subsection{Weak solution and characteristic flow} \label{sec:2.1}
In this subsection, we recall the weak formulation of the Cauchy problem \eqref{A-3}. Throughout the main noncompact analysis, we use $\mu_t$ for a measure-valued solution; whenever $\mu_t$ is absolutely continuous with respect to the Lebesgue measure $\di z$, we write 
\[ \mu_t(\di z)=f_t(z)\di z \]
and reserve $f_t$ for its density. Next, we recall the concepts of a density weak solution and a coupling of two measures. 
\begin{definition}\label{D2.1}
	\emph{(Density weak solution).}
	For $\tau\in(0,\infty]$, let
	$f\in {\mathcal C}([0,\tau);L^1(\mathbb{R}^{2d}))\cap L^\infty_{\rm loc}([0,\tau)\times\mathbb{R}^{2d})$ be a (density) weak solution to \eqref{A-3} with a nonnegative initial datum
	$f^{\mathrm{in}} \in(L^1\cap L_+^{\infty})(\mathbb{R}^{2d})$ if the following relations hold.
	\begin{enumerate}
		\item For every $\psi\in{\mathcal C}_c^1(\mathbb{R}^{2d})$,
		\[
		\mbox{the map}~~t\longmapsto\int_{\mathbb{R}^{2d}}\psi(z)f_t(z)\,\di z \quad \mbox{is continuous}. 
		\]
		\item For every $\zeta\in{\mathcal C}_c^1([0,\tau)\times\mathbb{R}^{2d})$ and every $t\in[0,\tau)$,
		\begin{align*}
		\begin{aligned}
			& \int_{\mathbb{R}^{2d}}\zeta(t,z)f_t(z)\,\di z - \int_{\mathbb{R}^{2d}}\zeta(0,z)f^{\mathrm{in}}(z)\,\di z \\
			& \hspace{2.5cm} =\int_0^t\int_{\mathbb{R}^{2d}}
			\Bigl(\partial_s\zeta+v\cdot\nabla_x\zeta+\nabla_v\zeta\cdot L[f_s]\Bigr)
			f_s(z)\,\di z\,\di s.
		\end{aligned}
		\end{align*}
	\end{enumerate}
\end{definition}
\begin{definition}\label{D2.2}
Let $\mu$ and $\nu$ be probability measures on $\R^d$. 
\begin{enumerate}
\item
A joint probability measure $\pi\in{\mathcal P}(\R^d\times\R^d)$  is called a coupling of $\mu$ and $\nu$ if its first and second marginals are $\mu$ and $\nu$, respectively.
\vspace{0.2cm}
\item
We denote the set of all such couplings by
\[
\Pi(\mu,\nu)
:=\left\{\pi\in{\mathcal P}(\R^d\times\R^d):
 ({\mathbb P}_1)_{\#}\pi=\mu,\quad  ({\mathbb P}_2)_{\#}\pi=\nu\right\},
\]
where ${\mathbb P}_1:~(x_1, x_2) \in {\mathbb R}^{2d}  \quad \mapsto \quad x_1 \in {\mathbb R}^d$ and  ${\mathbb P}_2:~(x_1, x_2) \in {\mathbb R}^{2d}~~\mapsto~~x_2 \in {\mathbb R}^d$ represent the orthogonal projections on the first and second argument spaces, respectively. 
\end{enumerate}
\end{definition}
Next, we describe a particle trajectory associated with the KMT model. In the literature, particle trajectory is also called ``bi-characteristics or simply characteristics". For $z = (x, v) \in {\mathbb R}^d \times {\mathbb R}^d$, we define a particle trajectory 
\[ (X(t), V(t)) : = (X(t:0,z), V(t:0,z)) \]
as a unique solution to the following Cauchy problem:
\begin{equation} \label{B-1}
	\begin{cases}
		\displaystyle \dot X(t)=V(t), \quad t > 0,  \vspace{6pt} \\
		\displaystyle \dot V(t)=L[f_t](X(t),V(t)),  \vspace{6pt}\\
		\displaystyle X(0)=x, \quad  V(0)=v,
	\end{cases}
\end{equation}
where the velocity alignment force $L[f_t]$ along the particle trajectory is given by the following relation:
\begin{equation}\label{B-2}
\displaystyle	L[f_t](X(t),V(t))=-\kappa\frac{\int_{\mathbb{R}^{2d}} \phi( |X(t) -x_{\star}|) \left(V(t)-v_{\star}\right)  f_t(z_{\star}) \di z_{\star}}{\int_{\mathbb{R}^{2d}} \phi( |X(t) -x_{\star}|) f_t(z_{\star}) \di z_{\star}}, \quad \di z_{\star} := \di x_{\star} \di v_{\star}.
\end{equation}
Note that $L[f_t]$ in \eqref{B-2} is continuous in $t$, locally Lipschitz continuous in the state variables $(X,V)$, and has at most linear growth in $V$ on every finite time interval. Thus, the standard Cauchy--Lipschitz theory yields the local existence and uniqueness of the characteristic system \eqref{B-1}. Moreover, the linear-growth estimate prevents finite-time blow-up and hence extends the characteristic trajectories globally in time. Consequently, for each fixed $t\geq 0$, the characteristic flow
\[
\Phi_t(z):=(X(t: 0,z),V(t: 0,z))
\]
is well defined. Whenever the forward and backward characteristic systems are uniquely solvable, $\Phi_t$ is a homeomorphism of the phase space. Moreover, we identify the initial density with the measure
$
\mu_0(\di z)
=
f^{\mathrm{in}}(z)\di z,
$
the solution is transported by the characteristic flow in the sense that
$
\mu_t
=
(\Phi_t)_{\#}\mu_0.$
Equivalently, we have
\begin{equation}\label{B-3}
	\int_{\mathbb{R}^{2d}}
	\psi(z)f_t(z)\,\di z
	=
	\int_{\mathbb{R}^{2d}}
	\psi\bigl(X(t,0,z),V(t,0,z)\bigr)
f^{\mathrm{in}}(z)\di z,
	\quad
	\forall\,\psi\in{\mathcal C}_b^1(\mathbb{R}^{2d}).
\end{equation}

 \subsection{Previous results}\label{sec:2.2}
In this subsection, we recall previous results in phase-spatially confined setting and existence of global weak solutions. For a nonnegative density $h$ on $\mathbb R^{2d}$, we define projected supports:
\[
\Omega_x(h):=\operatorname{{\mathbb P}}_x\operatorname{spt}h,
\qquad
\Omega_v(h):=\operatorname{{\mathbb P}}_v\operatorname{spt}h,
\]
and, if they are bounded, we also define radius and diameter functionals of the projected supports:
\[
\begin{cases}
\displaystyle \mathcal R_x(h) :=\sup_{x\in\Omega_x(h)}|x|, \quad 
\mathcal R_v(h) :=\sup_{v\in\Omega_v(h)}|v|,\vspace{6pt}\\
\displaystyle \mathcal D_x(h) :=\sup_{x,x_{\star}\in\Omega_x(h)}|x-x_{\star}|, \quad 
\mathcal D_v(h) :=\sup_{v,v_{\star}\in\Omega_v(h)}|v-v_{\star}|.
\end{cases}
\]
In what follows, we recall a set of previous results on the strong flocking and a global weak solution to \eqref{A-3} in a phase-spatially confined setting. For an initial datum $f^{\mathrm{in}}$ with a compact support, we set 
\[
 \Lambda_\phi(t) := \int_0^t \phi(s) \, \di s, \quad  {\mathcal D}^{\infty} :=  \Lambda_\phi^{-1}\left(\frac{\mathcal{D}_{v}(f^{\mathrm{in}})}{\kappa} + \Lambda_\phi(\mathcal{D}_{x}(f^{\mathrm{in}}))\right).
\]
\begin{proposition}\emph{(Strong flocking \cite{flocking1,flocking2}) }\label{P2.1}
	Let \(f\) be a global weak solution to \eqref{A-3} with compactly supported initial datum \(f^{\mathrm{in}}\):
\[ \mathcal{D}_{x}(f^{\mathrm{in}}) < +\infty \quad \mbox{and} \quad \mathcal{D}_{v}(f^{\mathrm{in}})< +\infty. \]
Then the following assertions hold.
\begin{enumerate}
\item
Velocity radius and diameter functionals are contractive:
	\[\mathcal{R}_{v}(f_t)\le \mathcal{R}_{v}(f^{\mathrm{in}}), \quad \mathcal{D}_{v}(f_t)\le \mathcal{D}_{v}(f^{\mathrm{in}}), \quad t > 0. \]
\item	
	If the communication weight function  and coupling strength satisfy 
	\[
\mathcal{D}_{v}(f^{\mathrm{in}})<\kappa \int_{\mathcal{D}_{x}(f^{\mathrm{in}})}^\infty \phi(r) \, \di r  ,
	\]
	then, strong flocking emerges exponentially fast:  
		\[
	\sup_{0 \leq t < \infty} \mathcal{D}_{x}(f_t) \leq \mathcal{D}^{\infty}, \quad \mathcal{D}_{v}(f_t) \leq \mathcal{D}_{v}(f^{\mathrm{in}}) e^{-\kappa\phi({\mathcal D}^{\infty})t}, \quad t > 0. \]
\end{enumerate}	
\end{proposition}	
\begin{remark}\label{R2.1}
	The previous results \cite{flocking1,MT,flocking2} rely on the fact that the communication function has a positive lower bound under compact support for any finite time $t$, which can obtain the flocking behavior of the solution. However, under noncompact position-velocity support, the above energy functional method fails since $\mathcal{D}_x(f^{\mathrm{in}})=\infty$ and $\mathcal{D}_v(f^{\mathrm{in}})=\infty$.
\end{remark}
Next, we recall an existence criterion for the KMT equation which is adapted from \cite{kinetic3}.
\begin{proposition}[Existence criterion \cite{kinetic3}]\label{P2.2}
Suppose the initial datum $f^{\mathrm{in}}$ satisfies
\[
f^{\mathrm{in}} \in(L^1\cap L_+^\infty)(\mathbb R^{2d}),
\quad
(|x|^2+|v|^2)f^{\mathrm{in}} \in L^1(\mathbb R^{2d}).
\]
If the regularized approximations $\{f^\varepsilon\}_{\varepsilon>0}$ used in the weak-solution construction satisfy the  uniform second-moment bound condition:~for every $T>0$,
\begin{equation}\label{B-4}
\sup_{\varepsilon>0}\sup_{0\le t\le T}\int_{\mathbb R^{2d}}(|x|^2+|v|^2)f_t^\varepsilon(z)\,\di z\le C_T,
\end{equation}
with $C_T<\infty$ independent of the regularization parameter, then one may pass to the limit in the approximation scheme and obtain a weak solution
$
f\in\mathcal C([0,T];L^1(\mathbb R^{2d}))\cap L_+^\infty([0,T]\times\mathbb R^{2d})
$
to \eqref{A-3}.
\end{proposition}
With the help of Proposition~\ref{P2.2}, we can derive global existence for \eqref{A-3} in the velocity-confined setting:  
\begin{proposition}\label{P2.3}
Suppose the initial datum $f^{\mathrm{in}}$ satisfies
\[
f^{\mathrm{in}} \in(L^1\cap L_+^\infty)(\mathbb R^{2d}),
\quad
|x|^2f^{\mathrm{in}} \in L^1(\mathbb R^{2d}),
\quad
\mathcal R_v(f^{\mathrm{in}})<\infty.
\]
Then the following assertions hold.
\begin{enumerate}
\item
There exists a global weak solution to \eqref{A-3} such that
\begin{equation}\label{B-5}
\mathcal R_v(f_t)\le\mathcal R_v(f^{\mathrm{in}}), \quad 
\mathcal D_v(f_t)\le\mathcal D_v(f^{\mathrm{in}}), \quad 
\int_{\mathbb R^{2d}}f_t(z)\di z
=\int_{\mathbb R^{2d}}f^{\mathrm{in}}(z)\di z, \quad t > 0.
\end{equation}
\item
For every $T>0$,
\begin{equation}\label{B-6}
\sup_{0\le t\le T}\int_{\mathbb R^{2d}}(|x|^2+|v|^2)f_t(z)\,\di z<\infty.
\end{equation}
\end{enumerate}
\end{proposition}
\begin{proof}
\noindent (1)~
The normalized alignment field is a convex average of velocities. The closed convex hull of the initial velocity support is forward invariant, and the velocity diameter is non-increasing. This gives the first two estimates in \eqref{B-5}. To justify conservation of total mass, test the weak equation with a standard family of cutoffs $\chi_R(z)$ satisfying $0\le \chi_R\le1$, $\chi_R\to1$ pointwise, and $|\nabla\chi_R|\lesssim R^{-1}$. Letting $R\to\infty$, we have
\[
\int_{\mathbb R^{2d}} f_t(z)\,\di z
=
\int_{\mathbb R^{2d}} f^{\mathrm{in}}(z)\,\di z.
\]

\noindent (2)~We set
\[
R_0:=\mathcal R_v(f^{\mathrm{in}}),
\quad
M_{x,2}(t):=\int_{\mathbb R^{2d}}|x|^2f_t(z)\,\di z.
\]
Since the initial total mass is finite and positive, we assume
it to one without loss of generality. By mass conservation and the
bound $|v|\le R_0$ on the velocity support, we obtain
\[
\int_{\mathbb R^{2d}}|v|^2f_t(z)\,\di z\le R_0^2.
\]
Using the weak equation with the standard cutoff approximation of the test function $|x|^2$, we obtain
\[
\frac{\di}{\di t}M_{x,2}(t)
=2\int_{\mathbb R^{2d}}x\cdot v\,f_t(z)\,\di z
\le2M_{x,2}(t)^{1/2}R_0, \quad \mbox{a.e.}~t > 0.
\]
Therefore, we have
\[
M_{x,2}(t)^{1/2}\le M_{x,2}(0)^{1/2}+R_0t.
\]
This yields
\begin{equation}\label{B-7}
\int_{\mathbb R^{2d}}(|x|^2+|v|^2)f_t(z)\,\di z
\le\bigl(M_{x,2}(0)^{1/2}+R_0t\bigr)^2+R_0^2.
\end{equation}
The same estimate holds uniformly for the regularized approximations in Proposition~\ref{P2.2}. Thus the estimate \eqref{B-4} follows on every finite interval. Proposition~\ref{P2.2} gives a global weak solution, and the support estimates pass to the limit. This verifies \eqref{B-5} and \eqref{B-6}.
\end{proof}
\begin{remark}\label{R2.2}
		We can use \eqref{B-2} and \eqref{B-5} to see that 
	\[ \Big |\mathcal{D}_{v}(f_t)-\mathcal{D}_{v}(f_s) \Big |\le 2\kappa\mathcal{D}_{v}(f^{\mathrm{in}})|t-s|. \]
	This implies that $\mathcal{D}_{v}(f_t)$ is Lipschitz continuous and therefore, it is differentiable almost everywhere.
\end{remark}

\section{Description of main results}\label{sec:3}
\setcounter{equation}{0}
In this section, we introduce the concept of $\alpha$-moderate communication function, and then we present three main results on the global well-posedness and weak flockings. 
\subsection{Moderate communication weight function} \label{sec:3.1} In this subsection, we present a new class of communication functions. 
\begin{definition}[$\alpha$-moderate communication function]\label{D3.1}
Let $\alpha$ be a nonnegative real number.  Then, a communication function $\phi: {\mathbb R}_+~\to~(0,1]$ is called \emph{$\alpha$-moderate} if $\phi$ satisfies the following structural properties:
	\begin{enumerate}
	\item
	 $\phi$ is bounded and non-increasing:
	\begin{equation} \label{C-1}
	\phi(0)=1 \quad \mbox{and} \quad (\phi(r_1) - \phi(r_2)) (r_1 - r_2) \leq 0, \quad r_1,~r_2 \geq 0.
	 \end{equation}
	\item
 There exists a constant $C_\phi\ge1$ such that
		\begin{equation}\label{C-2}
			M_\phi(s):=\sup_{r \geq 0}\frac{\phi(r)}{\phi(r+s)}
			\le C_\phi(1+s)^\alpha,
			\quad s\ge0.
		\end{equation}
		\item 
		There exists a fixed-scale comparability constant $d_\phi$:
		\begin{equation}\label{C-3}
			d_\phi:=\inf_{r\ge0}\frac{\phi(3r/2)}{\phi(r)}>0.
		\end{equation}
		\item 
		There exists a nonnegative constant $L_\phi\ge0$ such that
		\begin{equation}\label{C-4}
			|\log\phi(r)-\log\phi(s)|\le L_\phi|r-s|,
			\quad r,s\ge0.
		\end{equation}
	\end{enumerate}
\end{definition}
Next, we provide several comments on the $\alpha$-moderate communication function and the roles of these properties.
\begin{remark}\label{R3.1}
Note that the CS communication function $\eqref{A-2}_1$  is $\beta$-moderate. Recall that 
\begin{equation} \label{C-5}
\phi(r) :=\frac{1}{(1+r)^{\beta}}, \quad  r \geq 0, \quad \beta \geq 0.
\end{equation}
To see this, we take 
\[
	\alpha=\beta,
	\quad C_\phi=1,
	\quad d_\phi=(2/3)^\beta,
	\quad L_\phi=\beta.
	\]
\noindent (i)~ The explicit ansatz \eqref{C-5} yields the desired estimate \eqref{C-1}. \newline

\noindent (ii)~Since
\[
1+r+s\le(1+r)(1+s),
\]
we have
\[
\frac{\phi(r)}{\phi(r+s)}
=\left(\frac{1+r+s}{1+r}\right)^\beta
\le(1+s)^\beta.
\]
This verifies  \eqref{C-2} with $C_\phi=1$ and $\alpha=\beta$. \newline

\noindent (iii)~Again, we use $1+r \ge(2/3)(1+3r/2)$ to see
\[
\frac{\phi(3r/2)}{\phi(r)}
=\left(\frac{1+r}{1+3r/2}\right)^\beta
\ge(2/3)^\beta.
\]
\noindent (iv)~We use  
\[ 
\log \phi(r) = -\beta \log (1 + r)
\]
to see
\[
\left|\frac{\di}{\di r}\log\phi(r)\right|
=\frac{\beta}{1+r}\le\beta.
\]
This gives \eqref{C-4}.
\end{remark}
\begin{remark}\label{R3.2}
The relations in Definition~\ref{D3.1} play different roles in the analysis. More precisely, we comment on the relations in Definition \ref{D3.1} one by one. 
\begin{enumerate}
\item
The relations in \eqref{C-1} provide the basic positivity and monotonicity required for the normalized MT interaction to be well defined and for spatial distance estimates to be converted into lower bounds for the communication weight.
\vspace{0.1cm}
\item
The relations \eqref{C-2} control the deterioration of the communication strength under spatial translations by at most an algebraic factor; it is the main ingredient in deriving effective-region lower bounds and in estimating normalized averages of unbounded velocity observables by finite spatial-velocity moments. 
\vspace{0.1cm}
\item
The relation \eqref{C-3} gives a uniform comparison of the communication weight at two comparable spatial scales and prevents the normalized interaction from degenerating in the far-field part of the effective-region argument. 
\vspace{0.1cm}
\item
The relation \eqref{C-4} provides quantitative stability of the normalized kernel under bounded perturbations of the particle configuration and is used essentially in the direct fixed-point construction and uniqueness argument for the fully noncompact velocity regime.
\end{enumerate}
\end{remark}
\subsection{Global well-posedness} \label{sec:3.2}
In this subsection, we present the global Lagrangian well-posedness of \eqref{A-3}. 
\begin{proposition} 
\label{P3.1}
Suppose that system parameters, communication weight function and initial datum $\mu_0\in\mathcal P(\mathbb R^{2d})$ satisfy 
\begin{equation} \label{C-6}
 \kappa > 0, \quad  \alpha \in [0, 1], \quad  \phi~:~\mbox{$\alpha$-moderate}, \quad \int_{\mathbb R^{2d}} |z|^2\,\mu_0(\di z)<\infty.
\end{equation}
Then, the following assertions hold.
\begin{enumerate}
\item
The KMT equation \eqref{A-3} admits a unique global Lagrangian weak solution, in the class generated by remainders
\[
U\in C_{\rm loc}\bigl([0,\infty);L^\infty(\mu_0;\mathbb R^d)\bigr),
\]
of the form
\begin{equation} \label{C-7}
\mu_t=(X(t,\cdot),V(t,\cdot))_{\#}\mu_0.
\end{equation}
Moreover, for every $T>0$,
\begin{equation}\label{C-8}
\sup_{0\le t\le T}\int_{\mathbb R^{2d}}
\Bigl(|X(t,z)|^2+|V(t,z)|^2\Bigr)\,\mu_0(\di z)<\infty,
\end{equation} 
\item
If initial measure $\mu_0$ satisfies 
\[
\mu_0(\di z)=f^{\mathrm{in}}(z)\,\di z,
\quad
f^{\mathrm{in}} \in(L^1\cap L_+^\infty)(\mathbb R^{2d}),
\]
then $\mu_t$ remains absolutely continuous with respect to $\di z$ and its density $f_t$ is a weak solution in the sense of Definition~\ref{D2.1}.
\end{enumerate}
\end{proposition}
\begin{proof}
	Since the proof is very lengthy, we leave its detailed proof in Section \ref{sec:4}. However, for the readers' convenience, we sketch a proof  strategy as follows. 
	\begin{itemize}
		\item
		\textbf{Step A} (Closed equation for the bounded remainder):~
		We rewrite the characteristic system as a closed integral equation for
		the interaction-generated remainder $U$, after we separate the
		pointwise-unbounded initial velocity tail (see \eqref{D-9}).
		\vspace{.1cm}
		\item
		\textbf{Step B} (Normalized $L^2$--$L^\infty$ estimate):~
		Although $q_0$ itself may be unbounded, its normalized MT
		average admits an $L^\infty$ bound in terms of its $L^2$ norm and the
		pairwise spatial second moment (see Lemma \ref{L4.2} and Lemma \ref{L4.3}).
			\vspace{.1cm}
		\item
		\textbf{Step C} (Stability of the normalized averaging operator):~
		The logarithmic Lipschitz property of $\phi$ yields a quantitative
		stability estimate for the normalized averaging operator under bounded
		perturbations of the configuration. The unbounded initial tail cancels
		when two candidate configurations are compared, so the corresponding
		configuration difference remains bounded (See Lemma \ref{L4.4}).
			\vspace{.1cm}
		\item
		\textbf{Step D} (Finite-time bounds and global continuation):~
		A finite-time $L^\infty$ estimate for $U$, together with the induced
		control of the pairwise spatial second moment, rules out finite-time
		blow-up and permits the local solution to be continued globally.
	\end{itemize}
\end{proof}
\subsection{Weak flocking} \label{sec:3.3}
In this subsection, we present the weak support flocking and weak moment flocking of order $p$ without proofs, respectively.
\begin{theorem}[Exponential weak support flocking]
\label{T3.1}
Suppose that parameters, communication function and initial datum satisfy 
\[ \kappa > 0, \quad \alpha \in [0, 1], \quad \phi:~\mbox{$\alpha$-moderate},  \quad p \in [2, \infty), \quad \mathcal X_p(0)<\infty, \quad\mathcal D_v(0)<\infty, \]
and let $\mu_t=(X(t,\cdot),V(t,\cdot))_{\#}\mu_0$ be the global Lagrangian weak solution to \eqref{A-3} furnished by Proposition~\ref{P3.1}. Then  there exist positive constants $C>0$ and $\lambda>0$, depending only
on $\phi,~\kappa$ and $\mu_0$ such that
\begin{equation}\label{C-9}
	\sup_{t\ge0}\mathcal X_p(t)<\infty
	\quad \mbox{and} \quad 
	\mathcal D_v(t)\le Ce^{-\lambda t},
	\qquad t\ge0.
\end{equation}
\end{theorem}
\begin{proof}
Since the proof is also very lengthy, we leave the detailed proof in Section~\ref{sec:5}. For the reader's
	convenience, we summarize the main argument as follows.
\begin{itemize}
\item
		\textbf{Step A} (Construction of a time-varying effective-region):~
		Since the spatial support may be unbounded, the communication function
	 need not admit a uniform positive lower bound over the whole spatial support. Therefore, 
		we construct, at each time $t$, an effective spatial region
		carrying a fixed positive fraction of the total mass. The size of this
		region is controlled in terms of the pairwise spatial fluctuation
		$\mathcal X_p(t)$. See Lemma \ref{L5.1}.
		\vspace{0.1cm}
		\item
		\textbf{Step B} (Common-mass lower bound for the normalized kernels):~
		We use the relations for $\phi$  in
		Definition~\ref{D3.1} to show that every normalized
		MT interaction kernel assigns a common positive amount of
		weight to the effective region. More precisely, the common-mass
		coefficient admits a lower bound of the form
		\[
		\eta(t)
		\ge \frac{1}{2C_{\phi}}
\min\bigl\{6^{-\alpha},d_{\phi}\bigr\}
		\bigl(1+\mathcal X_p(t)\bigr)^{-\alpha}.
		\]
		See Lemma \ref{L5.2}.
		
		\vspace{0.1cm}
		\item
		\textbf{Step C} (Contraction of the velocity-support diameter):~
		The common-mass lower bound yields a Dobrushin-type contraction
		estimate for the normalized averaging operator. Consequently, the
		velocity-support diameter satisfies
		\begin{equation}
		\begin{cases} \label{C-10}
		\displaystyle D^+\mathcal X_p
		\le
		\mathcal D_v, \quad \mbox{a.e.,}~t > 0, \vspace{6pt}\\
		\displaystyle D^+\mathcal D_v
		\le
		- \frac{\kappa}{2C_{\phi}}
\min\bigl\{6^{-\alpha},d_{\phi}\bigr\}
		\bigl(1+\mathcal X_p \bigr)^{-\alpha}
		\mathcal D_v.
		\end{cases}
		\end{equation}
		See Lemma \ref{L5.3} and Lemma \ref{L5.4}.
		
		\vspace{0.1cm}
		\item
		\textbf{Step D} (Lyapunov closure and exponential alignment):~For $\alpha \in [0, 1]$, we use \eqref{C-10} and a nonlinear Lyapunov functional approach to derive 
		\[
		\sup_{t\ge0}\mathcal X_p(t)<\infty.
		\]
		Consequently, the coefficient in $\eqref{C-10}_2$:
		\[
		\frac{\kappa}{2C_{\phi}}
\min\bigl\{6^{-\alpha},d_{\phi}\bigr\}
		\bigl(1+\mathcal X_p(t)\bigr)^{-\alpha}
		\]
		has a strictly positive uniform lower bound. We substitute this bound into the velocity-diameter inequality $\eqref{C-10}_2$ to get 
		\[
		\mathcal D_v(t)
		\le
		Ce^{-\lambda t},
		\qquad t\ge0,
		\]
		for some constants $C,\lambda>0$. Hence the solution exhibits
		exponential weak support flocking. See Lemma \ref{L5.5}.
	\end{itemize}
\end{proof}
Recall
\[
m_0
:=
\int_{\mathbb R^{2d}}v\,\mu_0(\di x,\di v),
\qquad
q_0(z):=v-m_0,
\]
and let $\mu_t$ be the global Lagrangian solution constructed in
Proposition~\ref{P3.1}. Then, it satisfies 
\begin{equation}\label{C-11}
	V(t,z)
	=
	m_0+e^{-\kappa t}q_0(z)+U(t,z).
\end{equation}
We set 
\[
\mathcal D_U(t)
:=
\operatorname*{ess\,sup}_{z,z_{\star} \in {\mathbb R}^{2d}}
|U(t,z)-U(t,z_{\star})|.
\]
\begin{theorem}[Exponential weak moment flocking] \label{T3.2}
Suppose that parameters, communication weight function and initial datum satisfy 
\[ \kappa > 0, \quad \alpha \in [0, 1], \quad \phi:~\mbox{$\alpha$-moderate},  \quad p \in [2, \infty), \quad \mathcal X_p(0)<\infty, \quad \mathcal V_p(0)<\infty, \]
and let $\mu_t$ be the global Lagrangian weak solution furnished by Proposition~\ref{P3.1}. Then, there exist constants
$C>0$ and $\lambda>0$ such that
	\begin{equation}\label{C-12}
		\mathcal D_U(t)\le Ce^{-\lambda t}, \quad \sup_{t\ge0}\mathcal X_p(t)<\infty,
		\quad
		\mathcal V_p(t)\le Ce^{-\lambda t},
		\quad t\ge0.
\end{equation}
\end{theorem}
\begin{proof}
	The detailed proof can be found in Section~\ref{sec:6}, however for reader's
	convenience, we summarize the main argument in several steps.
	
	\begin{itemize}
		\item
		\textbf{Step A} (Decomposition of the unbounded velocity tail):~
		Using \eqref{C-11}, we write
		\[
		V(t,z)
		=
		m_0+e^{-\kappa t}q_0(z)+U(t,z).
		\]
		The possibly unbounded part of the initial velocity is therefore
		isolated in the explicit exponentially damped term
		$e^{-\kappa t}q_0(z)$, whereas the interaction-generated remainder
		$U$ is bounded on every finite time interval.
		
		\vspace{0.1cm}
		\item
		\textbf{Step B} (Estimate for the normalized average of the
			unbounded tail):~
		Although $q_0$ itself need not belong to $L^\infty(\mu_0)$, the
		algebraic moderation of the communication function and the finite
		$p$-moment ($p\ge2$) assumptions yield
		\[
		\|\mathscr T_{X(t)}q_0\|_{L^\infty(\mu_0)}
		\le
		C\|q_0\|_{L^2(\mu_0)}
		\bigl(1+\mathcal X_p(t)\bigr)^\alpha.
		\]
		Thus the contribution of the unbounded initial velocity tail enters
		the equation for $U$ only through a bounded forcing term. See Lemma \ref{L6.2}.
		
		\vspace{0.1cm}
		\item
		\textbf{Step C} (Common-mass contraction for the remainder):~We apply the time-varying effective-region argument to the
		normalized interaction kernels. This gives a common positive mass
		shared by all kernels and hence a contraction estimate for the
		oscillation
		\[
		\mathcal D_U(t)
		:=
		\operatorname*{ess\,sup}_{z,z_{\star}}
		|U(t,z)-U(t,z_{\star})|.
		\]
		Now, we introduce the auxiliary spatial scale
		\[
		\mathcal S_U(t)
		:=
		1+\mathcal X_p(0)
		+\frac1\kappa\mathcal V_p(0)
		+\int_0^t\mathcal D_U(s)\,\di s,
		\]
		to find 
		\[
		1+\mathcal X_p(t)\le \mathcal S_U(t),
		\qquad
		\mathcal S_U'(t)=\mathcal D_U(t).
		\]See Lemma \ref{L6.3}.
		
		\vspace{0.1cm}
		\item
		\textbf{Step D} (Reduction to a closed scalar differential system):~
		Combining the common-mass contraction with the estimate from
		\textbf{Step~B}, we obtain
		\begin{equation} 
		\begin{cases} \label{C-13}
		\displaystyle	D^+\mathcal D_U(t)
			\le
			-a\mathcal S_U(t)^{-\alpha}\mathcal D_U(t)
			+
			be^{-\kappa t}\mathcal S_U(t)^\alpha, \quad \mbox{a.e.,}~t > 0, \vspace{6pt}\\
		\displaystyle \mathcal S_U'(t)=\mathcal D_U(t),  \vspace{6pt}\\
		\displaystyle (\mathcal{D}_U, {\mathcal S}_U) \Big|_{t=0} = \Big (0, 1+\mathcal X_p(0)
		+\frac1\kappa\mathcal V_p(0) \Big),
		\end{cases}	
		\end{equation}
		for some constants $a,b>0$. Hence the fully noncompact flocking problem is reduced to a closed
		system for $\mathcal D_U$ and $\mathcal S_U$. See Lemma \ref{L6.4}.
		
		\vspace{0.1cm}
		\item
		\textbf{Step E} (Bootstrap of the spatial scale and exponential
			decay of $\mathcal D_U$):~
		For $0\le\alpha<1$, the initial linear growth estimate for
		$\mathcal S_U$ inserted into \eqref{C-13} yields
		stretched-exponential decay of $\mathcal D_U$. At the critical endpoint
		$\alpha=1$, an additional bootstrap first improves the growth of
		$\mathcal S_U$ from linear to sublinear; reinserting this improved bound into
		\eqref{C-13} then yields an integrable decay rate for $\mathcal D_U$.
		In both cases,
		\[
		\mathcal S_U'(t)=\mathcal D_U(t)
		\]
		implies $\sup_{t\ge0}\mathcal S_U(t)<\infty$. Substituting this uniform
		bound back into \eqref{C-13} upgrades the decay to
		\[
		\mathcal D_U(t)\le Ce^{-\lambda t}.
		\]
	 See Lemma \ref{L6.5}.
		
		\vspace{0.1cm}
		\item
		\textbf{Step F} (Exponential decay of the velocity $p$-fluctuation):~It follows from the decomposition \eqref{C-11}:
		\[
		V(t,z)-V(t,z_{\star})
		=
		e^{-\kappa t}
		\bigl(q_0(z)-q_0(z_{\star})\bigr)
		+
		U(t,z)-U(t,z_{\star}).
		\]
Then, we have
		\[
		\mathcal V_p(t)
		\le
		e^{-\kappa t}\mathcal V_p(0)
		+
		\mathcal D_U(t).
		\]
		Using the exponential decay obtained in \textbf{Step~E}, we conclude that
		\[
		\mathcal V_p(t)\le Ce^{-\lambda t},
		\qquad
		\sup_{t\ge0}\mathcal X_p(t)<\infty.
		\]
		Hence the solution exhibits exponential weak moment flocking. See Section \ref{sec:6.3}.
	\end{itemize}

\end{proof}

\begin{remark}\label{R3.3}
	We make several comments on Theorem~\ref{T3.1} and Theorem~\ref{T3.2}.
	\begin{enumerate}
		\item
		The two theorems describe genuinely different flocking regimes.
		When the velocity support is compact, the natural observable is the
		velocity-support diameter $\mathcal D_v$, and the normalized
		MT interaction yields exponential contraction of the
		whole velocity support. When the velocity support is noncompact,
		$\mathcal D_v(t)$ may remain infinite at every finite time, and the
		appropriate observable is instead the pairwise velocity fluctuation
		$\mathcal V_p$. Thus Theorem~\ref{T3.1} gives exponential weak
		support flocking, whereas Theorem~\ref{T3.2} gives exponential weak
		moment flocking.
		
		\vspace{0.1cm}
		\item
		Compared with the classical flocking theory for the KMT model with
		compactly supported initial data \cite{MT,flocking1,HWXMT,HWYMT,W-MT-2026,HJZ2025}, Theorem~\ref{T3.1} and Theorem~\ref{T3.2}
		remove the compactness assumption on the spatial distribution and,
		in the second theorem, also on the velocity distribution. In the
		compact-support setting, the communication weight admits a uniform
		positive lower bound over the whole support, so the alignment
		mechanism can be treated through a fixed global interaction scale.
		In the present noncompact setting, such a uniform lower bound is no
		longer available. The key replacement is a time-dependent effective
		region carrying a fixed positive fraction of the total mass, which
		provides a quantitative common-mass lower bound for the normalized
		interaction kernels. 
		\vspace{0.1cm}
		\item
		The present framework also differs essentially from the weak-flocking
		theory for symmetric Cucker--Smale-type models with noncompact
		support \cite{H-W-2025,H-W-2026,H-W-CS-2026,HaWangXue2025}. In the latter case, conservation of the average velocity
		and the associated dissipative structure provide a distinguished
		reference state and play a central role in controlling spatial and
		velocity fluctuations. For the KMT model, the normalization destroys
		momentum conservation, so no conserved average velocity is available
		to serve as a global reference frame. Our arguments therefore do not
		rely on momentum conservation. Instead, the normalized common-mass
		contraction is combined with the time-varying effective-region
		method, and, in the fully noncompact regime, with the decomposition
		of the velocity into the exponentially damped initial tail and a
		bounded interaction-generated remainder. This provides a different
		mechanism for obtaining weak flocking in the absence of the
		conservation structure available for the Cucker--Smale dynamics.
		\vspace{0.1cm}
		\item 
		For the particular power-law communication function
		\[
		\phi(r)=(1+r)^{-\beta},
		\]
		Definition~\ref{D3.1} holds with $\alpha=\beta$. Therefore,
		Theorem~\ref{T3.1} and Theorem~\ref{T3.2} can be applied to the fat-tail range with
		$0\le\beta\le1.$ In particular, the endpoint $p=\infty$
		corresponds to the compactly supported case \cite{flocking1,MT}, which is also covered by our approach. 
	\end{enumerate}
\end{remark}

\section{Lagrangian weak solution}
\label{sec:4}
\setcounter{equation}{0}
In this section, we provide a proof of Proposition \ref{P3.1} on the global Lagrangian well-posedness of the KMT equation \eqref{A-3} in the fully noncompact regime.

\subsection{Particle trajectories}\label{sec:4.1} In this subsection, we derive a closed system for particle trajectories $(X(t), V(t))$ introduced in Section \ref{sec:2.1}. In this case, the obvious difficulty is that the initial velocity fluctuation
\begin{equation} \label{D-1}
q_0(z):=v-m_0
\end{equation}
is not assumed to be bounded on the velocity support, so a characteristic flow based on a uniformly bounded velocity field is not directly available. The key observation is to separate the unbounded part of the velocity through the exact decomposition
\[
V(t,z)=m_0+e^{-\kappa t}q_0(z)+U(t,z),
\]
where the interaction-generated remainder $U$ will be shown to remain bounded on every finite time interval in the sequel. For a measurable configuration $X:\mathbb R^{2d}\to\mathbb R^d$, we set the normalization factor and kernel:
\[
\mathcal Z_X(z)
:=
\int_{\mathbb R^{2d}}
\phi\bigl(|X(z)-X(z_{\star})|\bigr)\,\mu_0(\di z_{\star}), \quad k_X(z,z_{\star})
:=
\frac{\phi(|X(z)-X(z_{\star})|)}{\mathcal Z_X(z)}.
\]
Then, we define the normalized averaging operator as follows.
\begin{equation} \label{D-2}
\mathscr T_Xh(z) :=\int_{{\mathbb R}^{2d}} k_X(z,z_{\star})h(z_{\star})\,\mu_0(\di z_{\star}).
\end{equation}
Note that it is linear and satisfies
\[
\mathscr T_X1= \int k_X(z,z_{\star}) \,\mu_0(\di z_{\star}) = 1.
\]
Next, we summarize the property of $\mathscr T_X$ and give representation of particle trajectories in the following lemma.
\begin{lemma} \label{L4.1}
Let $(X,V)$ be a candidate Lagrangian trajectory satisfying the characteristic relations, and set
$\mu_t=(X(t,\cdot),V(t,\cdot))_{\#}\mu_0$. Then the following assertions hold.
\begin{enumerate}
\item
The averaging operator $\mathscr T_X$ is a positive operator such that 
\begin{equation} \label{D-3}
\mathscr T_Xm_0=m_0, \quad  \|\mathscr T_Xh\|_{L^\infty(\mu_0)}\le \|h\|_{L^\infty(\mu_0)}. 
\end{equation}
\item
The forward particle trajectory $(X(t,z), V(t,z))$ issued from $z$ at time $0$ satisfies 
\begin{equation} 
\hspace{0.5cm}
\begin{cases} \label{D-4}
\displaystyle X(t,z)=x+m_0t+\frac{1-e^{-\kappa t}}{\kappa}q_0(z)
+\int_0^tU(s,z)\,\di s, \vspace{6pt}\\
\displaystyle V(t,z)
	=m_0+e^{-\kappa t}q_0(z)+U(t,z), \vspace{6pt}\\
\displaystyle U(t,z)
		= \kappa\int_0^t e^{-\kappa(t-s)}
		\mathscr T_{X(s)}U(s,\cdot)(z)\,\di s +\kappa\int_0^t e^{-\kappa(t-s)}e^{-\kappa s}
		(\mathscr T_{X(s)}q_0)(z)\,\di s.
\end{cases}
\end{equation}
\end{enumerate}
\end{lemma}
\begin{proof}
\noindent (1)~Since $\phi>0$ and $\mu_0$ is a probability measure, $\mathcal Z_X(z)$ is strictly positive; hence the operator $\mathscr T_X$ is positive. On the other hand, since we have $\mathscr T_X1=1$, 
\begin{equation}\label{D-5}
\|\mathscr T_Xh\|_{L^\infty(\mu_0)}\le \|h\|_{L^\infty(\mu_0)}
\end{equation}
for every bounded scalar- or vector-valued function $h$. \newline

\noindent (2)~Let $(X(t), V(t))$ be a particle trajectory. \newline

\noindent $\bullet$~(Derivation of $\eqref{D-4}_2$ and $\eqref{D-4}_3$): ~Note that $\mu_t=(X(t),V(t))_{\#}\mu_0$ is a weak solution to \eqref{A-3}, and the relations \eqref{B-1} and  \eqref{B-2} yield
\begin{equation} \label{D-6}
\bar v_{\mu_t}(X(t,z))
=\mathscr T_{X(t)}V(t,\cdot)(z), \quad 
\dot V(t,z)+\kappa V(t,z)
=\kappa\mathscr T_{X(t)}V(t,\cdot)(z).
\end{equation}
On the other hand, we use the variation of constants for $\eqref{D-6}_2$ to get 
\begin{equation} \label{D-7}
V(t,z)
=e^{-\kappa t}v
+\kappa\int_0^t e^{-\kappa(t-s)}
\mathscr T_{X(s)}V(s,\cdot)(z)\,\di s.
\end{equation}
Now, we introduce a new function $U = U(t,z)$ such that 
\begin{equation} \label{D-8}
V(t,z) = m_0 + e^{-\kappa t} q_0(z) + U(t,z).
\end{equation}
Next, we derive a self-consistent equation for $U$. Again, we use  \eqref{D-7}, \eqref{D-8}  and $ v = q_0(z) + m_0$ to see
\begin{align}
\begin{aligned} \label{D-9}
U(t,z) &= V(t,z) - m_0 - e^{-\kappa t} q_0(z) \\
&= e^{-\kappa t}v - (m_0 +  e^{-\kappa t} q_0(z)) + \kappa\int_0^t e^{-\kappa(t-s)}
\mathscr T_{X(s)}V(s,\cdot)(z)\,\di s \\
&= (e^{-\kappa t} - 1) m_0 +  \kappa\int_0^t e^{-\kappa(t-s)}
\mathscr T_{X(s)}V(s,\cdot)(z)\,\di s \\
&=  (e^{-\kappa t} - 1) m_0 + \kappa\int_0^t e^{-\kappa(t-s)} \mathscr T_{X(s)} \Big( m_0+e^{-\kappa s}q_0+U(s,\cdot) \Big)(z) \di s \\
&=  \kappa \int_0^t e^{-\kappa(t-s)}e^{-\kappa s} (\mathscr T_{X(s)}q_0)(z) \di s + \kappa\int_0^t e^{-\kappa(t-s)} \mathscr T_{X(s)} U(s,\cdot)(z) \di s,
\end{aligned}
\end{align}
where we used the relation $\mathscr T_Xm_0=m_0$ to  see the cancellation:
\begin{align}
\begin{aligned} \label{D-10}
 & (e^{-\kappa t} - 1) m_0 + \kappa\int_0^t e^{-\kappa(t-s)} \mathscr T_{X(s)} m_0 \di s \\
 & \hspace{0.2cm} =  (e^{-\kappa t} - 1) m_0 + \kappa m_0 \int_0^t e^{-\kappa(t-s)} \di s =  (e^{-\kappa t} - 1) m_0 + m_0 (1 - e^{-\kappa t}) = 0.
\end{aligned}
\end{align}
We combine \eqref{D-7}, \eqref{D-9} and \eqref{D-10} to get the relation $\eqref{D-4}_2$:
\begin{align*}
\begin{aligned}
	V(t,z)
	&= m_0 + e^{-\kappa t} q_0(z) \\
        &+\kappa\int_0^t e^{-\kappa(t-s)}e^{-\kappa s}
	(\mathscr T_{X(s)}q_0)(z)\,\di s +\kappa\int_0^t e^{-\kappa(t-s)}
	\mathscr T_{X(s)}U(s,\cdot)(z) \,\di s.
\end{aligned}
\end{align*}

\vspace{0.2cm}

\noindent $\bullet$~(Derivation of $\eqref{D-4}_1$): By defining relation \eqref{D-8}, we have
\[  U(0,z)=0, \quad z \in {\mathbb R}^{2d}. \]
It follows from \eqref{D-8} and characteristic equation that 
\begin{equation} \label{D-11}
\begin{cases}
\displaystyle {\dot X}(t,z) = V(t, z) = m_0 + e^{-\kappa t} q_0(z) + U(t,z), \quad t > 0, \vspace{6pt}\\
\displaystyle X(0,z) = x.
\end{cases}
\end{equation}
Now, we integrate \eqref{D-11} to find the desired estimate:
\begin{equation}\label{D-12}
X(t,z)=x+m_0t+\frac{1-e^{-\kappa t}}{\kappa}q_0(z)
+\int_0^tU(s,z)\,\di s, \quad t \geq 0.
\end{equation}
\end{proof}

\subsection{Preparatory lemmas} \label{sec:4.2}
In this subsection, we study several properties of the averaging operator $\mathscr T_X$ defined in \eqref{D-2} which will be used crucially in the well-posedness and weak flocking estimates. For later use, if $X\in L^2(\mu_0;\mathbb R^d)$, we set two-point spatial correlation functional:
\begin{equation}\label{D-13}
\mathfrak X_2[X]
:=\left(\iint_{\mathbb R^{4d}}|X(z)-X(z_{\star})|^2\,
\mu_0(\di z)\mu_0(\di z_{\star})\right)^{1/2}.
\end{equation}
In the following lemma, we show that the normalized averaging operator satisfies an $L^2$--$L^\infty$ estimate. 

\begin{lemma}\label{L4.2}
Let $\phi$ be $\alpha$-moderate with $0\le\alpha\le1$, and let $X\in L^2(\mu_0;\mathbb R^d)$ and $h\in L^2(\mu_0;\mathbb R^m)$ for some $m\ge1$. Then there exists a constant $C_{\phi,\alpha}=2^{\alpha+1}C_{\phi}>0$ such that
\begin{equation}\label{D-14}
\|\mathscr T_Xh\|_{L^\infty(\mu_0)}
\le C_{\phi,\alpha}\|h\|_{L^2(\mu_0)}
\bigl(1+\mathfrak X_2[X]\bigr)^\alpha.
\end{equation}
\end{lemma}
\begin{proof} We split the proof into two steps. \newline

\noindent $\bullet$~{\bf Step A}:~It follows from \eqref{D-2} that 
\begin{equation} \label{D-15}
|\mathscr T_Xh(z)| \leq \frac{1}{\mathcal Z_X(z)} \int_{{\mathbb R}^{2d}} \phi(|X(z)-X(z_{\star})|) |h(z_{\star})| \mu_0(\di z_{\star}). 
\end{equation}
To estimate the normalized kernel factor, we set 
\begin{align}
\begin{aligned} \label{D-16}
{\overline X} &:=\int_{\mathbb R^{2d}}X(z)\,\mu_0(\di z), \quad R_{2,X}:=1+\mathfrak X_2[X], \\
\mathcal C_X &:=\{z_{\star} \in {\mathbb R}^{2d}:~|X(z_{\star})-\bar X|\le R_{2,X}\}.
\end{aligned}
\end{align}
Then, we claim that 
\begin{equation} \label{D-17}
\frac{\phi(|X(z)-X(z_{\star})|)}{\mathcal Z_X(z)}
\le 2C_\phi\bigl(1+|X(z_{\star})-\bar X|+R_{2,X}\bigr)^\alpha.
\end{equation}
\noindent {\bf Derivation of \eqref{D-17}}:~Since $\mu_0$ is a probability measure, the standard variance identity gives
\begin{equation}\label{D-18}
\int_{\mathbb R^{2d}}|X(z)-\bar X|^2\,\mu_0(\di z)
=\frac12\mathfrak X_2[X]^2.
\end{equation}
Then, we use \eqref{D-18} and Chebyshev's inequality to get 
\begin{equation} \label{D-19}
\mu_0(\mathcal C_X^c)
\le \frac{\mathfrak X_2[X]^2}{2R_{2,X}^2}\le\frac12,
\quad
\mu_0(\mathcal C_X)\ge\frac12.
\end{equation}
On the other hand, we fix an observation label $z$, and set 
\[
r:=|X(z)-\bar X|,
\]
For $z_{\star}\in\mathcal C_X$,
\[
|X(z)-X(z_{\star})| \leq |X(z)- \overline{X}| + |\overline{X} -X (z_{\star})|  \le r+R_{2,X}.
\]
Since $\phi$ is non-increasing, we have
\begin{align}
\begin{aligned} \label{D-20}
\mathcal Z_X(z) &=\int_{{\mathbb R}^{2d}} \phi(|X(z)-X(z_{\star})|)\,\mu_0(\di z_{\star}) \geq \int_{\mathcal C_X} \phi(|X(z)-X(z_{\star})|)\,\mu_0(\di z_{\star})  \\
& \geq \phi(r+R_{2,X}) \int_{{\mathcal C}_X} \mu_0(\di z_{\star}) \ge\frac12\phi(r+R_{2,X}).
\end{aligned}
\end{align}
  For arbitrary $z_{\star}$, we also set
\[
a:=|X(z)-X(z_{\star})|,
\quad
s:=|X(z_{\star})-\bar X|.
\]
Then, the triangle inequality yields
\[ r = |X(z)-\overline{X}|  \le |X(z) - X(z_{\star})| + |X(z_\star) - \overline{X}| = a+s. \]
Hence, we have 
\[ r+R_{2,X}\le a+s+R_{2,X}. \] 
By monotonicity of $\phi$ and the moderation condition \eqref{C-2}, we obtain
\begin{align}
\frac{\phi(a)}{\mathcal Z_X(z)}\le 2\frac{\phi(a)}{\phi(r+R_{2,X})}\le 2\frac{\phi(a)}{\phi(a+s+R_{2,X})}\le 2C_\phi(1+s+R_{2,X})^\alpha.
\label{D-21}
\end{align}

\vspace{0.2cm}

\noindent $\bullet$~{\bf Step B}:~We combine \eqref{D-15} and \eqref{D-17} and the Cauchy-Schwarz inequality to get 
\begin{align}
\begin{aligned} \label{D-22}
|\mathscr T_Xh(z)|
&\le 2C_\phi\int |h(z_{\star})|
\bigl(1+|X(z_{\star})-\bar X|+R_{2,X}\bigr)^\alpha\,\mu_0(\di z_{\star})\\
&\le 2C_\phi\|h\|_{L^2(\mu_0)}
\left(\int\bigl(1+|X-\bar X|+R_{2,X}\bigr)^{2\alpha}\,\di\mu_0\right)^{1/2}.
\end{aligned}
\end{align}
On the other hand, since $0\le\alpha\le1$ and $\mu_0$ is a probability measure, we have
\begin{equation} \label{D-23}
\left(\int A^{2\alpha}\,\di\mu_0\right)^{1/2}
\le\left(\int A^2\,\di\mu_0\right)^{\alpha/2}
\qquad (A\ge1).
\end{equation}
We apply this with $A=1+|X-\bar X|+R_{2,X}$ and use \eqref{D-18} to find 
\begin{align}
\begin{aligned} \label{D-24}
\left(\int A^2\,\di\mu_0\right)^{1/2} &\le 1+R_{2,X}+\|X-\bar X\|_{L^2(\mu_0)}=2+\mathfrak X_2[X]+\frac1{\sqrt2}\mathfrak X_2[X] \\
&\le 2\bigl(1+\mathfrak X_2[X]\bigr).
\end{aligned}
\end{align}
Finally, we combine \eqref{D-22}, \eqref{D-23} and \eqref{D-24} to obtain
\[
|\mathscr T_Xh(z)|
\le 2^{\alpha+1}C_{\phi}\|h\|_{L^2(\mu_0)}
\bigl(1+\mathfrak X_2[X]\bigr)^{\alpha},
\]
which proves \eqref{D-14}.
\end{proof}
\begin{remark}
It follows from \eqref{C-6} that 
\[ q_0 \in L^2(\mu_0;\mathbb R^d). \]
Now, we take $h=q_0$ and use the result of Lemma \ref{L4.2} to see
\begin{equation}\label{D-25}
\|\mathscr T_Xq_0\|_{L^\infty(\mu_0)}
\le C_{\phi,\alpha}\|q_0\|_{L^2(\mu_0)}
\bigl(1+\mathfrak X_2[X]\bigr)^\alpha.
\end{equation}
\end{remark}
\vspace{0.2cm}
Next, we provide an estimate on the two-point spatial correlation $\mathfrak X_2$. For a given $U\in\mathcal C([0,T];L^\infty(\mu_0;\mathbb R^d))$, we set 
\[
X_U(t,z)=x+m_0t+\frac{1-e^{-\kappa t}}{\kappa}q_0(z)
+\int_0^tU(s,z)\,\di s.
\]
\begin{lemma} 
\label{L4.3}
The following estimate holds.
\begin{equation} \label{D-26}
\mathfrak X_2[X_U(t)] \le \mathcal X_2(0)+\frac1\kappa\mathcal V_2(0)
+2\int_0^t\|U(s)\|_{L^\infty(\mu_0)}\,\di s.
\end{equation}
\end{lemma}
\begin{proof} We first claim that 
\begin{equation} \label{D-27}
\mathfrak X_2[X_U(t)] \le \mathcal{X}_2(0)
+\frac{1-e^{-\kappa t}}{\kappa}\mathcal V_2(0)
+2\int_0^t\|U(s)\|_{L^\infty(\mu_0)} \di s.
\end{equation}
\noindent {\it Proof of \eqref{D-27}}:~For $z,z_{\star}\in\mathbb R^{2d}$,
\begin{align*}
X_U(t,z)-X_U(t,z_{\star})
=x-x_{\star}
+\frac{1-e^{-\kappa t}}{\kappa}\bigl(q_0(z)-q_0(z_{\star})\bigr)+\int_0^t\bigl(U(s,z)-U(s,z_{\star})\bigr)\,\di s.
\end{align*}
Since $q_0(z)-q_0(z_{\star})=v-v_{\star}$, Minkowski's inequality in $L^2(\mu_0\otimes\mu_0)$ gives
\begin{align}
\begin{aligned} \label{D-28}
\mathfrak X_2[X_U(t)]
&\le \mathcal X_2(0)
+\frac{1-e^{-\kappa t}}{\kappa}\mathcal V_2(0)\\
&\quad+\int_0^t
\left(\iint|U(s,z)-U(s,z_{\star})|^2\,\mu_0(\di z)\mu_0(\di z_{\star})\right)^{1/2}\di s.
\end{aligned}
\end{align}
On the other hand, note that 
\begin{align}
\begin{aligned} \label{D-29}
& \iint|U(s,z)-U(s,z_{\star})|^2\,\mu_0(\di z)\mu_0(\di z_{\star}) \\
& \hspace{1cm} \leq 4 \|U(s)\|^2_{L^\infty(\mu_0)} \iint \mu_0(\di z)\mu_0(\di z_{\star}) \leq 4 \|U(s)\|^2_{L^\infty(\mu_0)}.
\end{aligned}
\end{align}
We combine \eqref{D-28} and \eqref{D-29} to get \eqref{D-27}. Finally, we use $1-e^{-\kappa t} \leq 1$ to get the desired estimate \eqref{D-26}.
\end{proof}
Next, we study the stability of the normalized averaging operator. 
\begin{lemma} \label{L4.4}
Suppose that the communication weight function $\phi$ satisfies 
\begin{equation} \label{D-30}
			|\log\phi(r)-\log\phi(s)|\le L_\phi|r-s|,
			\quad r,s\ge0,
\end{equation}
for a nonnegative constant $L_\phi\ge0$, and let $X,Y:\mathbb R^{2d}\to\mathbb R^d$ satisfy
\[
\delta_{X,Y}:=\|X-Y\|_{L^\infty(\mu_0)}<\infty.
\]
Then, for every measurable scalar- or vector-valued $h$ for which the right-hand side below is finite,
\begin{equation}\label{D-31}
\|\mathscr T_Xh-\mathscr T_Yh\|_{L^\infty(\mu_0)}
\le \bigl(e^{4L_\phi\delta_{X,Y}}-1\bigr)
\|\mathscr T_Y|h|\|_{L^\infty(\mu_0)}.
\end{equation}
\end{lemma}
\begin{proof}
By the reverse triangle inequality, we have
\[
\begin{aligned}
	\Big |
	|X(z)-X(z_{\star})|
	-
	|Y(z)-Y(z_{\star})|
	\Big |
	\le
	|X(z)-Y(z)|
	+
	|X(z_{\star})-Y(z_{\star})|
	\le 2\delta_{X,Y}.
\end{aligned}
\]
Therefore, the logarithmic Lipschitz condition \eqref{D-30} gives
\[
\left|
\log\phi(|X(z)-X(z_{\star})|)
-
\log\phi(|Y(z)-Y(z_{\star})|)
\right|
\le 2L_\phi\delta_{X,Y}.
\]
Exponentiating the above inequality yields
\begin{equation}\label{D-32}
	e^{-2L_\phi\delta_{X,Y}}
	\le
	\frac{\phi(|X(z)-X(z_{\star})|)}
	{\phi(|Y(z)-Y(z_{\star})|)}
	\le
	e^{2L_\phi\delta_{X,Y}}.
\end{equation}
Recall that 
\[
\mathcal Z_X(z):=\int\phi(|X(z)-X(z_{\star})|)\,\mu_0(\di z_{\star}),
\qquad
\mathcal Z_Y(z):=\int\phi(|Y(z)-Y(z_{\star})|)\,\mu_0(\di z_{\star}).
\]
We integrate \eqref{D-32} in $z_{\star}$ to find 
\[
e^{-2L_\phi\delta_{X,Y}}\le\frac{\mathcal Z_X(z)}{\mathcal Z_Y(z)}\le e^{2L_\phi\delta_{X,Y}}.
\]
Therefore the normalized kernels
\[
k_X(z,z_{\star}):=\frac{\phi(|X(z)-X(z_{\star})|)}{\mathcal Z_X(z)} 
\quad \mbox{and} \quad 
k_Y(z,z_{\star}):=\frac{\phi(|Y(z)-Y(z_{\star})|)}{\mathcal Z_Y(z)}
\]
satisfy
\[
e^{-4L_\phi\delta_{X,Y}}\le\frac{k_X(z,z_{\star})}{k_Y(z,z_{\star})}\le e^{4L_\phi\delta_{X,Y}}.
\]
It follows that
\[
|k_X(z,z_{\star})-k_Y(z,z_{\star})|
\le\bigl(e^{4L_\phi\delta_{X,Y}}-1\bigr)k_Y(z,z_{\star}).
\]
Thus, we have
\[
|\mathscr T_Xh(z)-\mathscr T_Yh(z)| \le\int |k_X(z,z_{\star})-k_Y(z,z_{\star})|\,|h(z_{\star})|\,\mu_0(\di z_{\star}) \le\bigl(e^{4L_\phi\delta_{X,Y}}-1\bigr)\mathscr T_Y|h|(z),
\]
which verifies \eqref{D-31} after taking the essential supremum in $z$.
\end{proof}

\subsection{Local well-posedness}
\label{sec:4.3}
In this subsection, we study the local well-posedness of \eqref{A-3}. For this, we first consider a unique solvability of the self-consistent equation for $U$:
\begin{equation} \label{D-33}
U(t,z)
		= \kappa\int_0^t e^{-\kappa(t-s)}
		\mathscr T_{X(s)}U(s,\cdot)(z)\,\di s +\kappa\int_0^t e^{-\kappa(t-s)}e^{-\kappa s}
		(\mathscr T_{X(s)}q_0)(z)\,\di s.
\end{equation}	

\begin{proposition}[Local Lagrangian well-posedness]
\label{P4.1}
Suppose that the assumptions in Proposition~\ref{P3.1} hold. Then the following assertions hold.
\begin{enumerate}
\item
There exists $T_0>0$ such that the equation \eqref{D-33} admits a unique solution $U\in\mathcal C([0,T_0];L^\infty(\mu_0;\mathbb R^d))$. 
\vspace{0.2cm}
\item
For associated $X$ and $V$ in \eqref{D-4}, the push-forward measure \eqref{C-7} forms a local Lagrangian weak solution to \eqref{A-3}.
\end{enumerate}
\end{proposition}
\begin{proof}
(1) We consider two cases for existence and uniqueness one by one. \newline

\noindent $\bullet$~Step A (Unique solvability in small time interval):~In what follows, we use the contraction mapping theorem for the solvability of \eqref{D-33}. For this, we introduce a closed ball in some function space, and nonlinear map, and then we show that the constructed nonlinear map is continuous from the ball to the same ball. More precisely, we set 
\[
\mathfrak B_T:=\mathcal C([0,T];L^\infty(\mu_0;\mathbb R^d)),
\quad
\|U\|_{\mathfrak B_T}:=\sup_{0\le t\le T}\|U(t)\|_{L^\infty(\mu_0)}.
\]
For $U\in\mathfrak B_T$, we define $X_U$ and $(\mathscr F U)$:~for $t > 0,~z \in {\mathbb R}^{2d}$, 
\begin{align}
\begin{aligned} \label{D-34}
X_U(t,z) &:=x+m_0t+\frac{1-e^{-\kappa t}}{\kappa}q_0(z)+\int_0^tU(s,z)\,\di s, \\
(\mathscr F U)(t,z) &:=\kappa\int_0^t e^{-\kappa(t-s)}\mathscr T_{X_U(s)}U(s,\cdot)(z)\,\di s+\kappa\int_0^t e^{-\kappa(t-s)}e^{-\kappa s}\mathscr T_{X_U(s)}q_0(z)\,\di s.
\end{aligned}
\end{align}

\vspace{0.2cm}

\noindent $\diamond$~Step A.1:  For a fixed $M>0$, we set
\[
\mathbb B_M:=\{U\in\mathfrak B_T:\|U\|_{\mathfrak B_T}\le M\}.
\]
Then, we claim that $\mathscr F$ maps $\mathbb B_M$ into $\mathbb B_M$. To see this, we may restrict to $T\le1$. By Lemma~\ref{L4.3}, for $U\in\mathbb B_M$ and $0\le t\le T$,
\begin{equation} \label{D-35}
\mathfrak X_2[X_U(t)] \leq  \mathcal X_2(0)+\frac1\kappa\mathcal V_2(0)
+2\int_0^t\|U(s)\|_{L^\infty(\mu_0)}\,\di s \leq \mathcal X_2(0)+\frac1\kappa\mathcal V_2(0)+2M.
\end{equation}
Hence \eqref{D-25} and \eqref{D-35} give the uniform bound:
\begin{align}
\begin{aligned} \label{D-36}
\|\mathscr T_{X_U} q_0\|_{L^\infty} &\leq  C_{\phi,\alpha}\|q_0\|_{L^2(\mu_0)}
\bigl(1+\mathfrak X_2[X_U]\bigr)^\alpha \\
& \leq  C_{\phi,\alpha}\|q_0\|_{L^2(\mu_0)} \Big[  1 +   \mathcal X_2(0)+\frac1\kappa\mathcal V_2(0)+2M       \Big ]^{\alpha} =: Q_M.
\end{aligned}
\end{align}
On the other hand, it follows from \eqref{D-3} and \eqref{D-34} that 
\begin{align}
\begin{aligned} \label{D-37}
&\|(\mathscr F U)(t) \|_{L^\infty(\mu_0)} \\
& \leq \kappa\int_0^t e^{-\kappa(t-s)} \| \mathscr T_{X_U(s)}U(s) \|_{L^\infty(\mu_0)} \di s+\kappa\int_0^t e^{-\kappa(t-s)}e^{-\kappa s} \| \mathscr T_{X_U(s)}q_0 \|_{L^\infty(\mu_0)} \di s \\
&  \leq \kappa M \int_0^t \underbrace{e^{-\kappa(t-s)}}_{\leq 1} ds + \kappa Q_M \int_0^t \underbrace{e^{-\kappa(t-s)}e^{-\kappa s}}_{\leq 1} \di s \\
&  \leq \kappa T(M+Q_M).
\end{aligned}
\end{align}
Taking the supremum of the above estimate over $t\in[0,T]$, we obtain 
\[
\|\mathscr F U\|_{\mathfrak B_T}\le\kappa T(M+Q_M).
\]
Now, we choose a sufficiently small $T>0$ such that
\begin{equation}\label{D-38}
\kappa T(M+Q_M)\le M.
\end{equation}
This yields
\[  \mathscr F(\mathbb B_M)\subset\mathbb B_M. \]
The integrands in \eqref{D-34} are uniformly bounded in $L^\infty$. More precisely, by
\eqref{D-5} and \eqref{D-36}, there exists $C_{M,T}>0$ such that
\[
\|(\mathscr F U)(t)-(\mathscr F U)(s)\|_{L^\infty(\mu_0)}
\le C_{M,T}|t-s|,
\qquad 0\le s,t\le T.
\]
Hence $\mathscr F U\in\mathfrak B_T$.\vspace{.2cm}

\noindent $\diamond$~Step A.2: The map $\mathscr F$ is a contraction in norm $\| \cdot \|_{\mathfrak B_T}$: there exists a positive constant $L= L(M, T)< 1$ such that 
\begin{equation} \label{D-39}
\|\mathscr F U-\mathscr F\widetilde U\|_{\mathfrak B_T} \le L \|U-\widetilde U\|_{\mathfrak B_T}.
\end{equation}
{\it (Derivation of \eqref{D-39})}:~Let $U,\widetilde U\in\mathbb B_M$. Then, it follows from \eqref{D-34} and a similar argument in \eqref{D-37} that 
\begin{align}
\begin{aligned} \label{D-40}
&\|(\mathscr F U)(t) - (\mathscr F \widetilde{U})(t) \|_{L^\infty(\mu_0)} \\
& \hspace{1cm} \leq \kappa \int_0^t e^{-\kappa(t-s)} \| \mathscr T_{X_U(s)}U(s) -  \mathscr T_{X_{\widetilde{U}(s)}} \widetilde{U}(s) \|_{L^\infty(\mu_0)} \di s \\
& \hspace{1cm} +\kappa\int_0^t e^{-\kappa(t-s)}e^{-\kappa s} \| \mathscr T_{X_U(s)}q_0   - \mathscr T_{X_{\widetilde{U}}(s)}q_0  \|_{L^\infty(\mu_0)} \di s \\
& \hspace{1cm} =: I_{11} + I_{12}.
\end{aligned}
\end{align}
\noindent $\clubsuit$~(Estimate of $I_{11}$):~We set
\[
\Delta_U:=\|U-\widetilde U\|_{\mathfrak B_T}.
\]
Since the terms containing $x$, $m_0$, and $q_0$ are identical for two candidates,
\begin{equation}\label{D-41}
\delta_t := \|X_U(t)-X_{\widetilde U}(t)\|_{L^\infty}
\le\int_0^t\|U(s)-\widetilde U(s)\|_{L^\infty}\,\di s
\le t\Delta_U.
\end{equation}
For $U,\widetilde U\in \mathbb B_M$, we have
\[
\Delta_U\le 2M,
\quad
\delta_t\le t\Delta_U\le 2Mt\le 2MT.
\]
Now, we are ready to estimate $\| \mathscr T_{X_U}U-\mathscr T_{X_{\widetilde U}}\widetilde U \|_{L^\infty(\mu_0)}$. For this, we use  the decomposition:
\[
\begin{aligned}
	\mathscr T_{X_U(t)}U(t)
	-\mathscr T_{X_{\widetilde U}(t)}\widetilde U(t)
	=\mathscr T_{X_U(t)}
	\big(U(t)-\widetilde U(t)\big)
	+\Big(
	\mathscr T_{X_U(t)}
	-\mathscr T_{X_{\widetilde U}(t)}
	\Big)\widetilde U(t)
\end{aligned}
\]
to get
\begin{align}
\begin{aligned} \label{D-42}
\| \mathscr T_{X_U}U-\mathscr T_{X_{\widetilde U}}\widetilde U \|_{L^\infty(\mu_0)} &\leq \| \mathscr T_{X_U}(U-\widetilde U) \|_{L^\infty(\mu_0)}
+ \| (\mathscr T_{X_U}-\mathscr T_{X_{\widetilde U}})\widetilde U \|_{L^\infty(\mu_0)} \\
&=: I_{111} + I_{112}.
\end{aligned}
\end{align}
Below, we estimate the terms $I_{11i},~i=1,2$ one by one. \newline

\noindent $\clubsuit$~(Estimate of $I_{111}$):~By \eqref{D-5}, one has 
\begin{equation} \label{D-43}
I_{111} \le
\|U(t)-\widetilde U(t)\|_{L^\infty}
\le \Delta_U.
\end{equation}
\noindent $\clubsuit$~(Estimate of $I_{112}$):~We use Lemma~\ref{L4.4} with
\[ X=X_U(t), \quad Y=X_{\widetilde U}(t), \quad h=\widetilde U(t) \]
to see
\begin{equation} \label{D-44}
	I_{112} \le
	\left(e^{4L_\phi\delta_t}-1\right)
	\left\|
	\mathscr T_{X_{\widetilde U}(t)}
	|\widetilde U(t)|
	\right\|_{L^\infty}.
\end{equation}
Next, we estimate the factors in the right-hand side of \eqref{D-44} as follows. \newline

\noindent For the second factor, we use the positivity of $\mathscr T_X$, $\mathscr T_X1=1$, and
$\|\widetilde U(t)\|_{L^\infty}\le M$ to get 
\begin{equation} \label{D-45}
\left\|
\mathscr T_{X_{\widetilde U}(t)}
|\widetilde U(t)|
\right\|_{L^\infty}
\le M.
\end{equation}
For the first factor, we use 
\[  e^r-1\le re^r, \quad \delta_t\le t\Delta_U, \quad \mbox{and} \quad \Delta_U\le2M \]
to obtain
\begin{equation} \label{D-46}
	e^{4L_\phi\delta_t}-1
	\le
	4L_\phi\delta_t e^{4L_\phi\delta_t}\le
	4L_\phi t\Delta_U
	e^{4L_\phi t\Delta_U}\le
	4L_\phi t e^{8L_\phi MT}\Delta_U.
\end{equation}
We collect all the estimates \eqref{D-44}, \eqref{D-45} and \eqref{D-46} to find 
\begin{equation} \label{D-47}
I_{112} \leq 4L_\phi t e^{8L_\phi MT}M\Delta_U.
\end{equation}
Finally, we combine \eqref{D-42}, \eqref{D-43} and \eqref{D-47} to get 
\[
\left\|
\mathscr T_{X_U(t)}U(t)
-\mathscr T_{X_{\widetilde U}(t)}\widetilde U(t)
\right\|_{L^\infty}
\le
\Delta_U
+
4L_\phi t e^{8L_\phi MT}M\Delta_U.
\]
\noindent $\clubsuit$~(Estimate of $I_{12}$):~ Similar to the previous case, we use Lemma~\ref{L4.4} with $h=q_0$ to get 
\[
\begin{aligned}
	\left\|
	\mathscr T_{X_U(t)}q_0
	-\mathscr T_{X_{\widetilde U}(t)}q_0
	\right\|_{L^\infty}\le
	\left(e^{4L_\phi\delta_t}-1\right)
	\left\|
	\mathscr T_{X_{\widetilde U}(t)}|q_0|
	\right\|_{L^\infty}.
\end{aligned}
\]
Since $\widetilde U\in \mathbb B_M$, the spatial estimate used in
\eqref{D-36} applies to $X_{\widetilde U}$ as well. Hence,
by Lemma~\ref{L4.2} applied to $|q_0|$, we have
\[
\left\|
\mathscr T_{X_{\widetilde U}(t)}|q_0|
\right\|_{L^\infty}
\le Q_M.
\]
Therefore, we have
\begin{equation} \label{D-48}
\left\|
\mathscr T_{X_U(t)}q_0
-\mathscr T_{X_{\widetilde U}(t)}q_0
\right\|_{L^\infty}
\le
4L_\phi t e^{8L_\phi MT}Q_M\Delta_U.
\end{equation}
Now, we substitute \eqref{D-48} into \eqref{D-40} and bound the exponential factors by one to find 
\begin{equation}\label{D-49}
\|\mathscr F U-\mathscr F\widetilde U\|_{\mathfrak B_T}
\le\left[\kappa T+4\kappa L_\phi e^{8L_\phi MT}(M+Q_M)T^2\right]\Delta_U.
\end{equation}
After we choose sufficiently small $T$ if necessary, the coefficient in brackets is strictly smaller than one. Hence, the map $\mathscr F$ is a strict contraction on $\mathbb B_M$. Banach's fixed-point theorem yields a unique $U\in\mathbb B_M$ satisfying \eqref{D-33}.  Since every $\mathfrak B_T$-solution satisfies $U(0)=0$ and is continuous as an $L^\infty(\mu_0)$-valued map, any two solutions $U,\widetilde U\in\mathfrak B_T$ belong, after restricting to a sufficiently short interval $[0,T_0]$, to a common ball $\mathbb B_M$. The contraction estimate implies
$$
U=\widetilde U
\qquad\text{on }[0,T_0].
$$

\noindent $\bullet$~Step B (Extension of small time interval to a finite-time interval):  To extend the uniqueness to the whole interval $[0,T]$, suppose that
\[ U=\widetilde U \quad \mbox{on $[0,t_0]$}. \]
Then $U(t_0)=\widetilde U(t_0)$, and the
remainder equation can be restarted at $t_0$. Since both solutions are
bounded in $C([0,T];L^\infty(\mu_0))$, they lie in a common bounded ball
on a sufficiently short interval $[t_0,t_0+\tau]$, where the same
contraction argument applies. Hence we have
\[ U=\widetilde U \quad \mbox{on}~[t_0,t_0+\tau]. \]
Iterating this argument yields
$$
U=\widetilde U
\qquad\text{on }[0,T].
$$
Thus uniqueness holds in the full class $\mathfrak B_T$.\vspace{.1cm}

\noindent (2) We use the definition of $U$, \eqref{D-8} and \eqref{D-12} to derive the desired corresponding local Lagrangian weak solution.
\end{proof}
Next, we provide the finite bound for remainder term $U$.
\begin{lemma}[Finite-time bound for the remainder]
\label{L4.5}
Let $U$ be the maximal local solution furnished by Proposition~\ref{P4.1}, defined on $[0,T_{\max})$. Then, the following assertions hold.
\begin{enumerate}
\item
For every $T<T_{\max}$,
\begin{equation}\label{D-50}
\sup_{0\le t\le T}\|U(t)\|_{L^\infty(\mu_0)}<\infty,
\qquad
\sup_{0\le t\le T}\mathcal X_2(t)<\infty.
\end{equation}
\item
If $T_{\max}<\infty$, the bounds above remain uniform as $T\uparrow T_{\max}$.
\end{enumerate}
\end{lemma}
\begin{proof}
\noindent (1)~We set 
\[
M_U(t):=\|U(t)\|_{L^\infty(\mu_0)},
\quad
A_{\rm init}:=\mathcal X_2(0)+\frac1\kappa\mathcal V_2(0).
\]
Then, we use Lemma~\ref{L4.3} to get 
\begin{equation}\label{D-51}
\mathfrak X_2(t)\le A_{\rm init}+2\int_0^tM_U(s)\,\di s.
\end{equation}
We take the $L^\infty$-norm in \eqref{D-4} and use \eqref{D-5} to find 
\[
M_U(t)
\le\kappa\int_0^t e^{-\kappa(t-s)}M_U(s)\,\di s
+\kappa e^{-\kappa t}\int_0^t\|\mathscr T_{X(s)}q_0\|_{L^\infty}\,\di s.
\]
We define
\[
N_U(t):=e^{\kappa t}M_U(t), \qquad
g(t):=\|\mathscr T_{X(t)}q_0\|_{L^\infty}.
\]
Then, we have
\[
N_U(t)\le
\kappa\int_0^t N_U(s)\,\di s
+\kappa\int_0^t g(s)\,\di s.
\]
By the integral form of Grönwall's inequality,
\[
N_U(t)
\le
\kappa\int_0^t e^{\kappa(t-s)}g(s)\,\di s.
\]
Therefore, since $M_U(t)=e^{-\kappa t}N_U(t)$, we have
\begin{equation}\label{D-52}
M_U(t)\le\kappa\int_0^t e^{-\kappa s}
\|\mathscr T_{X(s)}q_0\|_{L^\infty}\,\di s.
\end{equation}
By Lemma~\ref{L4.2} and \eqref{D-51}, we get
\[
\begin{aligned}
	\|\mathscr T_{X(s)}q_0\|_{L^\infty}
	&\le C_{\phi,\alpha}\|q_0\|_{L^2(\mu_0)}
	\bigl(1+\mathcal X_2(s)\bigr)^\alpha \\
	&\le C_{\phi,\alpha}\|q_0\|_{L^2(\mu_0)}
	\left(
	1+A_{\rm init}
	+2\int_0^s M_U(r)\,\di r
	\right)^\alpha .
\end{aligned}
\]
Substituting this estimate into \eqref{D-52}, we obtain
\begin{equation}\label{D-53}
	M_U(t)
	\le C_{\rm rem}\int_0^t e^{-\kappa s}
	\left(1+A_{\rm init}+2\int_0^sM_U(r)\,\di r\right)^\alpha\di s,
\end{equation}
where
\[
C_{\rm rem}:=\kappa C_{\phi,\alpha}\|q_0\|_{L^2(\mu_0)}.
\]
Define
\[
H_U(t):=1+A_{\rm init}+2\int_0^tM_U(s)\,\di s.
\]
Since $H_U\ge1$ and $0\le\alpha\le1$, we have 
\[ H_U^\alpha\le H_U. \]
Dropping the harmless factor $e^{-\kappa s}$ from \eqref{D-53}, we have
\[
M_U(t)\le C_{\rm rem}\int_0^tH_U(s)\,\di s,
\quad
H_U'(t)=2M_U(t)\le2C_{\rm rem}\int_0^tH_U(s)\,\di s.
\]
We set 
\[
K_U(t):=H_U(t)+\int_0^tH_U(s)\,\di s.
\]
Then, we have
\[
K_U'(t)=H_U'(t)+H_U(t)
\le 2C_{\rm rem}\int_0^tH_U(s)\,\di s+H_U(t)
\le(1+2C_{\rm rem})K_U(t).
\]
Thus Grönwall's inequality gives
\begin{equation}\label{D-54}
H_U(t)\le K_U(t)\le H_U(0)e^{(1+2C_{\rm rem})t},
\qquad 0\le t<T_{\max}.
\end{equation}
Returning to \eqref{D-53}, we obtain, for every finite $T<T_{\max}$,
\[
\sup_{0\le t\le T}M_U(t)
\le C_{\rm rem}T H_U(0)e^{(1+2C_{\rm rem})T}<\infty.
\]
Together with \eqref{D-51}, this verifies \eqref{D-50}. Importantly, the right-hand sides depend only on $T$ and the initial data, not on the distance from $T$ to $T_{\max}.$ \newline

\noindent (2)~If  $T_{\max}<\infty$, the same estimates remain to be bounded as $T\uparrow T_{\max}$.
\end{proof}

\subsection{Proof of Proposition~\ref{P3.1}}
\label{sec:4.4}
In what follows, we split the proof into four steps. 

\vspace{0.1cm}

\noindent $\bullet$~ \textbf{Step A} (Global continuation):~Let $[0,T_{\max})$ be the maximal interval of existence of the local solution. We claim that 
\begin{equation} \label{D-55}
T_{\max} = \infty.
\end{equation}
{\it Proof of \eqref{D-55}}:~Suppose the contrary holds, i.e., 
\[ T_{\max}<\infty. \]
Now, we use  Lemma~\ref{L4.5} to find 
\[
\sup_{0\le t<T_{\max}}\|U(t)\|_{L^\infty}<\infty,
\qquad
\sup_{0\le t<T_{\max}}\mathcal X_2(t)<\infty.
\]
Let $t_0<T_{\max}$. Then, for $t\ge t_0$, the remainder equation can be written as
\begin{equation}\label{D-56}
\begin{aligned}
U(t)
=e^{-\kappa(t-t_0)}U(t_0)
+\kappa\int_{t_0}^t e^{-\kappa(t-s)}\mathscr T_{X(s)}U(s)\,\di s+\kappa\int_{t_0}^t e^{-\kappa(t-s)}e^{-\kappa s}\mathscr T_{X(s)}q_0\,\di s,
\end{aligned}
\end{equation}
and
\begin{equation}\label{D-57}
X(t,z)=X(t_0,z)+m_0(t-t_0)
+\frac{e^{-\kappa t_0}-e^{-\kappa t}}{\kappa}q_0(z)
+\int_{t_0}^tU(s,z)\,\di s.
\end{equation}
When two candidate extensions with the same history up to \(t_0\) are compared, all terms in \eqref{D-57} except the last one are canceled. The restarted equation has the same contraction structure, with an additional known forcing term \(e^{-\kappa(t-t_0)}U(t_0)\). Since this term is uniformly bounded, the same fixed-point argument as in Proposition~\ref{P4.1} applies on \([t_0,t_0+\tau]\), with constants depending only on the uniform bounds above and the initial moments. Hence \(\tau>0\) can be chosen uniformly for \(t_0\) sufficiently close to \(T_{\max}\), which extends the solution beyond \(T_{\max}\), a contradiction. Thus we have \[T_{\max}=\infty.\]

\medskip
\noindent $\bullet$~
\textbf{Step B} (Propagation of the second phase-space moment):~Fix $T>0$ and we set
\[
M_T:=\sup_{0\le t\le T}\|U(t)\|_{L^\infty}<\infty.
\]
Then, it follows from $\eqref{D-4}_2$ that 
\begin{equation}\label{D-58}
\|V(t,\cdot)\|_{L^2(\mu_0)}
\le |m_0|+\|q_0\|_{L^2(\mu_0)}+M_T,
\qquad 0\le t\le T.
\end{equation}
Similarly, the estimate \eqref{D-12} gives
\begin{equation}\label{D-59}
\|X(t,\cdot)\|_{L^2(\mu_0)}
\le \|x\|_{L^2(\mu_0)}+|m_0|T
+\frac1\kappa\|q_0\|_{L^2(\mu_0)}+TM_T.
\end{equation}
Therefore, we have
\[
\sup_{0\le t\le T}\int_{\mathbb R^{2d}}
\bigl(|X(t,z)|^2+|V(t,z)|^2\bigr)\,\mu_0(\di z)<\infty,
\]
which is \eqref{C-8}. \newline

\noindent $\bullet$~{\bf Step C} (Global uniqueness):~Let $(X,V)$ and $(\widetilde X,\widetilde V)$ be two global Lagrangian solutions in the class stated in Proposition~\ref{P3.1}, with the same initial datum, and let $U,\widetilde U$ denote their remainders. By the definition of this class, both remainders belong to $C([0,T];L^\infty(\mu_0))$ on every finite interval $[0,T]$. Starting at $t=0$, the contraction estimate \eqref{D-49} gives equality on a short interval. If two solutions agree up to some time $t_0$, the restarted equations \eqref{D-56}--\eqref{D-57} and the same contraction argument give equality on a further interval $[t_0,t_0+\tau]$. Iterating finitely many times on each compact time interval yields
\[
U=\widetilde U,
\quad
X=\widetilde X,
\quad
V=\widetilde V
\]
on $[0,\infty)$. This proves uniqueness. \newline

\noindent $\bullet$ \textbf{Step D} (Verification of the weak formulation):~The fixed-point identity $\eqref{D-4}_3$ implies that, for $\mu_0$-almost every $z$, $U(\cdot,z)$ is locally absolutely continuous and satisfies
\[
\dot U(t,z)
=-\kappa U(t,z)+\kappa\mathscr T_{X(t)}U(t,\cdot)(z)
+\kappa e^{-\kappa t}\mathscr T_{X(t)}q_0(z)
\]
for almost every $t>0$. We use $\eqref{D-4}_1$ and $\eqref{D-4}_2$ to obtain
\[
\dot X(t,z)=V(t,z),
\quad
\dot V(t,z)=\kappa\bigl(\mathscr T_{X(t)}V(t,\cdot)(z)-V(t,z)\bigr), \quad \mbox{a.e.}~t > 0.
\]
Define
\[
\mu_t:=(X(t,\cdot),V(t,\cdot))_{\#}\mu_0.
\]
By the push-forward identity,
\[
\mathscr T_{X(t)}V(t,\cdot)(z)=\bar v_{\mu_t}(X(t,z)).
\]
Moreover, \eqref{D-58}, \eqref{D-59}, and the continuity of $U$ in $L^\infty$ imply that the map $t\mapsto(X(t),V(t))$ is continuous in $L^2(\mu_0)$ on every finite interval. Hence $t\mapsto\mu_t$ is narrowly continuous.\vspace{.1cm}

Let $\psi\in\mathcal C_c^1([0,\infty)\times\mathbb R^{2d})$. For $\mu_0$-almost every $z$, the chain rule along the characteristic trajectory gives
\begin{align*}
\frac{\di}{\di t}\psi(t,X(t,z),V(t,z))
=\partial_t\psi(t,X,V)+V\cdot\nabla_x\psi(t,X,V)+\kappa\bigl(\bar v_{\mu_t}(X)-V\bigr)\cdot\nabla_v\psi(t,X,V).
\end{align*}
On every finite time interval, Lemma~\ref{L4.2}, \eqref{D-5}, and Lemma~\ref{L4.5} give
\[
\bar v_{\mu_t}(X(t,z))
=m_0+e^{-\kappa t}\mathscr T_{X(t)}q_0(z)
+\mathscr T_{X(t)}U(t)(z)
\]
with the right-hand side uniformly bounded in the observation label. Together with the $L^2$ bound for $V$, this justifies integration of the chain rule with respect to $\mu_0$ and in time. We obtain
\begin{align*}
&\int_{\mathbb R^{2d}}\psi(t,x,v)\,\mu_t(\di x,\di v)
-\int_{\mathbb R^{2d}}\psi(0,x,v)\,\mu_0(\di x,\di v)\\
&\qquad\qquad=\int_0^t\int_{\mathbb R^{2d}}
\Bigl[\partial_s\psi+v\cdot\nabla_x\psi
+L[\mu_s](x,v)\cdot\nabla_v\psi\Bigr]\,\mu_s(\di x,\di v)\,\di s.
\end{align*}
Thus $\mu_t$ is a global weak measure-valued solution of \eqref{A-3}. \newline

Finally, we suppose $\mu_0=f^{\mathrm{in}}\,\di x\,\di v$ with $f^{\mathrm{in}}\in L^1\cap L_+^\infty$. We first justify the local Lipschitz regularity of the Eulerian alignment field. Fix $T>0$ and a compact set $K\subset\mathbb R^d$. By the spatial second-moment bound, one can choose $R=R_{K,T}>0$ such that
\[
\inf_{0\le t\le T}\mu_t\bigl(B_R\times\mathbb R^d\bigr)\ge\frac12.
\]
Hence, with
\[
\mathcal Z_t(x):=\int\phi(|x-x_{\star}|)\,\mu_t(\di x_{\star},\di v_{\star}),
\]
monotonicity gives
\[
\inf_{0\le t\le T}\inf_{x\in K}\mathcal Z_t(x)
\ge\frac12\phi\bigl(R+\sup_{x\in K}|x|\bigr)>0.
\]
Moreover, if $x,y\in K$ and $|x-y|\le1$, then \eqref{C-4} and the reverse triangle inequality imply
\[
|\phi(|x-x_{\star}|)-\phi(|y-x_{\star}|)|
\le L_\phi e^{L_\phi}|x-y|\,\phi(|y-x_{\star}|).
\]
Using the finite first velocity moment, which follows from \eqref{C-8}, this estimate applies both to $\mathcal Z_t$ and to
\[
\mathcal N_t(x):=\int\phi(|x-x_{\star}|)v_{\star}\,\mu_t(\di x_{\star},\di v_{\star}).
\]
Since $\bar v_{\mu_t}=\mathcal N_t/\mathcal Z_t$ and $\mathcal Z_t$ is uniformly bounded away from zero on $[0,T]\times K$ and $K$ is arbitrary, it follows that $x\mapsto\bar v_{\mu_t}(x)$ is locally Lipschitz, uniformly on compact time intervals. Hence the phase-space vector field
\[
\mathbf b_t(x,v):=\bigl(v,\kappa(\bar v_{\mu_t}(x)-v)\bigr)
\]
is locally Lipschitz in $(x,v)$. Its divergence in phase space is
\[
\operatorname{div}_{x,v}\mathbf b_t=-\kappa d.
\]
By the standard theory of flows generated by time-dependent locally Lipschitz vector fields, the associated flow is locally bi-Lipschitz. Moreover, Liouville's formula gives
\[
\frac{\di}{\di t}J_t(z)
=
\bigl(\operatorname{div}_{x,v}\mathbf b_t\bigr)(\Phi_t(z))J_t(z)
=
-\kappa d\,J_t(z),
\qquad
J_0(z)=1,
\]
and therefore $J_t(z)=e^{-\kappa dt}$. The change-of-variables formula gives
\[
f_t(\Phi_t(z))=e^{\kappa dt}f^{\mathrm{in}}(z),
\qquad
\|f_t\|_{L^\infty}\le e^{\kappa dt}\|f^{\mathrm{in}}\|_{L^\infty}.
\]
Thus $\mu_t$ remains absolutely continuous. To justify the strong $L^1$ continuity in time, define
\[
P_tg:=e^{\kappa dt}g\circ\Phi_t^{-1}.
\]
Then $\|P_tg\|_{L^1}=\|g\|_{L^1}$. For $g\in C_c^\infty(\mathbb R^{2d})$, continuity of
$\Phi_t^{\pm1}$ on compact sets and the change-of-variables formula imply
$\|P_tg-P_sg\|_{L^1}\to0$ as $t\to s$. Approximating
$f^{\mathrm{in}}$ in $L^1$ by functions in $C_c^\infty$ and using the $L^1$-isometry of $P_t$ yields
\[
f\in C([0,T];L^1(\mathbb R^{2d}))
\cap L_+^\infty([0,T]\times\mathbb R^{2d})
\]
for every finite $T$. The density $f$ satisfies the weak formulation in
Definition~\ref{D2.1}. This completes the proof.

\section{Weak support flocking}\label{sec:5}
\setcounter{equation}{0}
In this section, we provide a proof of Theorem~\ref{T3.1}. Throughout this section, the push-forward 
$\mu_t=(X(t,\cdot),V(t,\cdot))_{\#}\mu_0$ of the initial measure $\mu_0$ denotes a global Lagrangian weak
solution appearing in Proposition~\ref{P3.1} for a fixed $p\ge1$. Recall that 
\begin{align*}
\begin{aligned}
\mathcal{X}_p(t)
&:=
\left(
\iint_{\mathbb R^{4d}}
|X(t,z)-X(t,z_{\star})|^p
\,\mu_0(\di z)\mu_0(\di z_{\star})
\right)^{1/p}, \\
\mathcal D_v(t) &:=
\operatorname*{ess\,sup}_{z,z_{\star}}
|V(t,z)-V(t,z_{\star})|.
\end{aligned}
\end{align*}
As outlined in Theorem~\ref{T3.1}, we first construct a time-varying effective-region carrying at least one half
of the total mass. Second, the algebraic moderateness of $\phi$ converts
this effective region into a common lower bound for all normalized MT
kernels. Third, this common mass gives a Dobrushin-type contraction of the
velocity-support diameter, while $\dot X=V$ controls the growth of
$\mathcal X_p$. Finally, a nonlinear Lyapunov functional closes the two
estimates and yields uniform spatial cohesion and exponential velocity support
alignment.

\subsection{Effective region and common mass}\label{sec:5.1}
In this subsection, we begin with a purely geometric consequence of the pairwise spatial
$p$-moment. It replaces the bounded-support localization used in the
classical compact-support theory.

\begin{lemma}[Time-dependent half-mass effective region]\label{L5.1}
For $p \geq 1$, let $X:\mathbb R^{2d}\to\mathbb R^d$ be a measurable map such that $\mathfrak X_p[X]< \infty$. Then there exists a label $z_c\in\mathbb R^{2d}$ such that
\begin{align}
\begin{aligned} \label{E-1}
& (i)~\int_{\mathbb R^{2d}}
|X(z_{\star})-X(z_c) |^p\,\mu_0(\di z_{\star})
\le \mathfrak X_p[X]^p. \\
& (ii)~\mu_0 \Big( \Big \{ z_{\star}\in\mathbb R^{2d}:
|X(z_{\star})- X(z_c) |\le 2^{1/p}\mathfrak X_p[X]  \Big \} \Big ) \ge\frac12.
\end{aligned}
\end{align}
\end{lemma}
\begin{proof} 
\noindent (i)~We set 
\[
F_X(z):=
\int_{\mathbb R^{2d}}
|X(z)-X(z_{\star})|^p\,\mu_0(\di z_{\star}).
\]
Since the integrand is nonnegative, Tonelli's theorem gives
\[
\begin{aligned}
	\int_{\mathbb R^{2d}}F_X(z)\,\mu_0(\di z)
	&=
	\int_{\mathbb R^{2d}}
	\int_{\mathbb R^{2d}}
	|X(z)-X(z_{\star})|^p
	\,\mu_0(\di z_{\star})\,\mu_0(\di z)
	\\
	&=
	\iint_{\mathbb R^{4d}}
	|X(z)-X(z_{\star})|^p\,
	\mu_0(\di z)\mu_0(\di z_{\star})
	=
	\mathfrak X_p[X]^p.
\end{aligned}
\]
\noindent $\bullet$~{\bf Case 1}:~Suppose that
\[ 0 < \mathfrak X_p[X]<\infty. \]
This yields 
\[
F_X(z)<\infty \quad
\text{for }\mu_0\text{-almost every }z.
\]
Moreover, because $\mu_0$ is a probability measure, there exists
$z_c$ in this full-measure set such that
\[
F_X(z_c)\le \mathfrak X_p[X]^p.
\]
Indeed, if this were not the case, then
\[
F_X(z)>\mathfrak X_p[X]^p
\qquad
\text{for }\mu_0\text{-almost every }z,
\]
and therefore
\[
\int_{\mathbb R^{2d}}F_X(z)\,\mu_0(\di z)
>
\mathfrak X_p[X]^p
\int_{\mathbb R^{2d}}\mu_0(\di z)
=
\mathfrak X_p[X]^p.
\]
This contradicts
\[
\int_{\mathbb R^{2d}}F_X(z)\,\mu_0(\di z)
=
\mathfrak X_p[X]^p.
\]
Hence there exists $z_c\in\mathbb R^{2d}$ such that
\[
\int_{\mathbb R^{2d}}
|X(z_c)-X(z_{\star})|^p\,
\mu_0(\di z_{\star})
\le
\mathfrak X_p[X]^p,
\]
which verifies $\eqref{E-1}_1.$ \newline

\noindent $\bullet$~{\bf Case 2}:~Suppose that 
\[  \mathfrak X_p[X]=0. \]
Then this  implies
\[ X(z_{\star})= X(z_c) \quad \mbox{for $\mu_0$-almost every $z_{\star}$}, \]
and therefore $\eqref{E-1}_1$ is immediate. \newline

\noindent (ii)~Suppose that
$\mathfrak X_p[X]>0$. By Chebyshev's inequality,
\[
\mu_0\bigl(
|X-X(z_c)|>2^{1/p}\mathfrak X_p[X]
\bigr) \le
\frac{1}{2\mathfrak X_p[X]^p}
\int_{\mathbb R^{2d}}|X-X(z_c)|^p\,\di\mu_0
\le\frac12.
\]
This verifies the second estimate.
\end{proof}
Recall that for a configuration $X$,  we have the normalized kernel density
\begin{equation}\label{E-2}
k_X(z,z_{\star})
:=
\frac{\phi(|X(z)-X(z_{\star})|)}{\mathcal Z_X(z)}, \quad \mathscr T_Xh(z)
=
\int_{\mathbb R^{2d}}
k_X(z,z_{\star})h(z_{\star})\,\mu_0(\di z_{\star}).
\end{equation}
In the next lemma, we study estimates on common-mass minorization and Dobrushin contraction which follows from the key consequence of the normalized MT structure. We set an effective region ${\mathcal E}_X$:
\[
{\mathcal E}_X := \Big \{ z_{\star}\in\mathbb R^{2d}:
|X(z_{\star})- X(z_c) |\le 2^{1/p}\mathfrak X_p[X]  \Big \}.
\]
Then, it follows from \eqref{E-1} that 
\[ \mu_0({\mathcal E}_X) \geq \frac{1}{2}.
\]
\begin{lemma}\label{L5.2}
Suppose that parameters, kernel and $X: {\mathbb R}^{2d} \to {\mathbb R}^d$ satisfy
\[  \alpha \in [0, 1], \quad p \geq 1, \quad \phi~:~\mbox{$\alpha$-moderate}, \quad \mathfrak X_p[X]<\infty. \]
Then, the following assertions hold.
\begin{enumerate}
\item
For $\mu_0$-almost every $z$ and every $z_{\star}\in\mathcal E_X$ outside a $\mu_0$-null set, we have
\begin{equation}\label{E-3}
k_X(z,z_{\star})
\ge
\frac{1}{C_{\phi}}
\min\bigl\{6^{-\alpha},d_{\phi}\bigr\} \bigl(1+\mathfrak X_p[X]\bigr)^{-\alpha},
\end{equation}
where $C_{\phi}$ is a positive constant appearing in \eqref{C-2}.
\vspace{0.2cm}
\item
Any two normalized kernels have a common submeasure of mass at least $\eta_X$:
\[ \eta_X := \frac{1}{2C_{\phi}}
\min\bigl\{6^{-\alpha},d_{\phi}\bigr\}  \bigl(1+\mathfrak X_p[X]\bigr)^{-\alpha}. \]
Moreover, for every bounded measurable
$h:\mathbb R^{2d}\to\mathbb R^m$, we have
\begin{equation}\label{E-4}
\operatorname*{ess\,sup}_{z,\widetilde z}
|\mathscr T_Xh(z)-\mathscr T_Xh(\widetilde z)|
\le
(1-\eta_X)
\operatorname*{ess\,sup}_{z_{\star},\widetilde z_{\star}}
|h(z_{\star})-h(\widetilde z_{\star})|.
\end{equation}
\end{enumerate}
\end{lemma}
\begin{proof}
\noindent (1)~By Lemma~\ref{L5.1}, we have
\begin{equation}\label{E-5}
\int_{\mathbb R^{2d}}|X(z_{\star})-X(z_c)|^p\,\mu_0(\di z_{\star})
\le  (\mathfrak X_p[X])^p,
\qquad
\mu_0(\mathcal E_X)\ge\frac12.
\end{equation}
Next, we estimate \eqref{E-2} in the near-field
and far-field regimes separately.  For $\mu_0$-almost every $z$ and every $z_{\star}\in\mathcal E_X$ outside a $\mu_0$-null set, we consider two cases:
\begin{align*}
\begin{aligned}
& |X(z)-X(z_c)| \le 2 (1+  2^{1/p}\mathfrak X_p[X]):~\mbox{near-field near $z_c$}, \\
& |X(z)-X(z_c)| > 2 (1+  2^{1/p}\mathfrak X_p[X]):~\mbox{far-field away from $z_c$}.
\end{aligned}
\end{align*}
\vspace{0.1cm}

\noindent $\bullet$~Case 1 (Near-field regime around $z_c$):~Suppose $z$ lies in the near-field regime around $z_c$ and $z_{\star}\in\mathcal E_X$. Then, we have
\[
|X(z)-X(z_c)| \le 2 (1+  2^{1/p}\mathfrak X_p[X]), \quad |X(z_{\star})- X(z_c) |\le 2^{1/p}\mathfrak X_p[X].
\]
These imply
\begin{align*}
\begin{aligned}
|X(z)-X(z_{\star})| &\le |X(z)-X(z_c)|+|X(z_{\star})- X(z_c)| \\
& \le  2 (1+  2^{1/p}\mathfrak X_p[X]) + 2^{1/p}\mathfrak X_p[X]  \le 3  (1+  2^{1/p}\mathfrak X_p[X]).
\end{aligned}
\end{align*}
Since $\phi$ is non-increasing and the denominator in
\eqref{E-2} is at most $1$, we obtain
\begin{equation} \label{E-6}
k_X(z,z_{\star})
\ge\phi \Big (3(1+  2^{1/p}\mathfrak X_p[X]) \Big).
\end{equation}
In \eqref{C-2}, we take 
\[ r=0 \quad \mbox{and} \quad s=3 (1+  2^{1/p}\mathfrak X_p[X]) \]
and use $\phi(0)=1$ to see
\begin{equation} \label{E-7}
\phi \Big(3(1+  2^{1/p}\mathfrak X_p[X]) \Big)
\ge
\frac{1}{C_{\phi}(1+3(1+  2^{1/p}\mathfrak X_p[X]))^{\alpha}}.
\end{equation}
Since $p\ge1$, $2^{1/p}\le2$, and hence
\begin{equation} \label{E-8}
1+3 (1+  2^{1/p}\mathfrak X_p[X])
=4+3  (2^{1/p}\mathfrak X_p[X])
\le4+6 \mathfrak X_p[X]
\le6(1+ \mathfrak X_p[X]).
\end{equation}
We combine \eqref{E-6}, \eqref{E-7} and \eqref{E-8} together with $0<\mathcal Z_X(z)\le1$ to find
\begin{equation}\label{E-9}
k_X(z,z_{\star})
\ge
\frac{6^{-\alpha}}{C_{\phi}}
(1+ \mathfrak X_p[X])^{-\alpha}.
\end{equation}
\vspace{0.2cm}

\noindent $\bullet$~Case 2 (Far-field regime away from $z_c$):~Suppose that $z$ belongs to a far-field regime away from $z_c$ and  $z_{\star}\in\mathcal E_X$:
\[
|X(z)-X(z_c)| > 2 (1+  2^{1/p}\mathfrak X_p[X]) \quad \mbox{and} \quad |X(z_{\star})- X(z_c) |\le 2^{1/p}\mathfrak X_p[X].
\]
These yield
\begin{align}
\begin{aligned} \label{E-10}
|X(z)-X(z_{\star})| &\le |X(z)-X(z_c)|+|X(z_{\star})- X(z_c)| \\
& \le |X(z)-X(z_c)|  + 2^{1/p}\mathfrak X_p[X] \le  \frac{3}{2} |X(z)-X(z_c) |.
\end{aligned}
\end{align}
Again we use the non-increasing property of $\phi$, \eqref{E-10}  and \eqref{C-3} to find 
\begin{equation}\label{E-11}
\phi(|X(z)-X(z_{\star})|)
\ge
\phi \Big( \frac{3}{2} |X(z)-X(z_c)| \Big)
\ge d_{\phi}\phi (|X(z)-X(z_c)|).
\end{equation}
We use triangle inequality,  non-increasing property of $\phi$ and \eqref{C-2} to see
\begin{align}
\begin{aligned} \label{E-12}
\frac{\phi(|X(z)-X(z_{\star})|)}{\phi( |X(z)-X(z_c) |)}
& \le
\frac{\phi(|X(z)-X(z_{\star})|)}{\phi(|X(z)-X(z_{\star})|+|X(z_{\star})-X(z_c)|)} \\
& \le
C_{\phi}(1+|X(z_{\star})-X(z_c)|)^{\alpha}.
\end{aligned}
\end{align}
We use \eqref{E-12} to get 
\begin{align}
\begin{aligned} \label{E-13}
 Z_X(z) &=\int_{\mathbb R^{2d}}
\phi(|X(z)-X(z_{\star})|)\,\mu_0(\di z_{\star}) \\
& \le
C_{\phi}\phi(|X(z)-X(z_c)|)
\int_{\mathbb R^{2d}}
(1+|X(z_{\star})-X(z_c)|)^{\alpha}\,\mu_0(\di z_{\star}).
\end{aligned}
\end{align}
Because $0\le\alpha\le1$ and $p\ge1$, Jensen's inequality and
\eqref{E-5} imply
\begin{align*}
\int(1+|X-X(z_c)|)^{\alpha}\,\di\mu_0
&\le
\left(1+\int|X-X(z_c)|\,\di\mu_0\right)^{\alpha}\\
&\le
\left(1+
\left(\int|X-X(z_c)|^p\,\di\mu_0\right)^{1/p}
\right)^{\alpha}
\le(1+ \mathfrak X_p[X])^{\alpha}.
\end{align*}
Combining this estimate with \eqref{E-11} and
\eqref{E-13} yields
\begin{equation}\label{E-14}
k_X(z,z_{\star})=\frac{\phi(|X(z)-X(z_{\star})|)}{\mathcal Z_X(z)}
\ge
\frac{d_{\phi}}{C_{\phi}}
(1+ \mathfrak X_p[X])^{-\alpha}.
\end{equation}
Finally, we combine \eqref{E-9} and \eqref{E-14} to get the desired estimate \eqref{E-3}.  \newline

\noindent (2)~Define the finite measure
\[
\sigma_X(\di z_{\star})
:=
\frac{1}{C_{\phi}}
\min\bigl\{6^{-\alpha},d_{\phi}\bigr\} (1+  \mathfrak X_p[X])^{-\alpha}
\mathbbm 1_{\mathcal E_X}(z_{\star})\,\mu_0(\di z_{\star}).
\]
By \eqref{E-3}, $\sigma_X$ is dominated by every
probability kernel
\[
P_X(z,\di z_{\star})
:=k_X(z,z_{\star})\mu_0(\di z_{\star}).
\]
Moreover, we use the definition of $\eta_X$ to see
\[
\omega_X:=\sigma_X(\mathbb R^{2d})
=\frac{1}{C_{\phi}}
\min\bigl\{6^{-\alpha},d_{\phi}\bigr\}(1+  \mathfrak X_p[X])^{-\alpha}\mu_0(\mathcal E_X)
\ge\eta_X.
\]
If $\omega_X=1$, then $\sigma_X\le P_X(z,\cdot)$ and both
measures have total mass one, so that
\[
P_X(z,\cdot)=\sigma_X
\]
for every observation label $z$. Hence $\mathscr T_Xh$ is constant
in $z$, and the relation \eqref{E-4} follows. Note that $\sigma_X\le P_X(z,\cdot)$ for every observation label $z$,
the measure $P_X(z,\cdot)-\sigma_X$ is nonnegative and has total mass
$1-\omega_X$. \newline

\noindent If $\omega_X<1$, we may define
\[
P_z^r
:=
\frac{P_X(z,\cdot)-\sigma_X}{1-\omega_X},
\]
which is a probability measure. Therefore, we have
\[
P_X(z,\cdot)
=
\sigma_X+(1-\omega_X)P_z^r,
\]
and similarly
\[
P_X(\widetilde z,\cdot)
=
\sigma_X+(1-\omega_X)P_{\widetilde z}^r.
\]
Therefore, we have
\[
\mathscr T_Xh(z)-\mathscr T_Xh(\widetilde z)
=
(1-\omega_X)
\left(
\int h\,\di P_z^r
-
\int h\,\di P_{\widetilde z}^r
\right).
\]
Since $P_z^r$ and $P_{\widetilde z}^r$ are probability measures,
\[
	\left|
	\int h\,\di P_z^r
	-
	\int h\,\di P_{\widetilde z}^r
	\right|
	=
	\left|
	\iint
	\bigl(h(z_\star)-h(\widetilde z_\star)\bigr)
	\,P_z^r(\di z_\star)
	P_{\widetilde z}^r(\di \widetilde z_\star)
	\right| \le
	\operatorname*{ess\,sup}_{z_\star,\widetilde z_\star \in {\mathbb R}^{2d}}
	|h(z_\star)-h(\widetilde z_\star)|.
\]
Thus, we have
\[
|\mathscr T_Xh(z)-\mathscr T_Xh(\widetilde z)|
\le
(1-\omega_X)
\operatorname*{ess\,sup}_{z_\star,\widetilde z_\star}
|h(z_\star)-h(\widetilde z_\star)|.
\]
Since $\omega_X\ge\eta_X$, we conclude that
\[
|\mathscr T_Xh(z)-\mathscr T_Xh(\widetilde z)|
\le
(1-\eta_X)
\operatorname*{ess\,sup}_{z_\star,\widetilde z_\star}
|h(z_\star)-h(\widetilde z_\star)|.
\]
Finally, we take the essential supremum over $z,\widetilde z$ to get the desired estimate \eqref{E-4}.
\end{proof}
\begin{remark}\label{R5.1}
Note that we provide a moment counterpart of the ``\textit{active sets}'' introduced in \cite{MT}. The main difference is that, instead of using  the compactness of the spatial support, the active region is controlled by the pairwise spatial moment.
\end{remark}
\subsection{Velocity contraction and spatial growth}\label{sec:5.2}
In this subsection, we derive a system of dissipative differential inequalities (SDDI) for $\mathcal X_p$ and $\mathcal D_v$:
\begin{equation} \label{E-15}
\begin{cases}
\displaystyle D^+\mathcal X_p \le\mathcal D_v, \quad \mbox{a.e.}~~t > 0, \vspace{6pt}\\
\displaystyle D^+\mathcal D_v
\le
- \frac{\kappa}{2C_{\phi}}
\min\bigl\{6^{-\alpha},d_{\phi}\bigr\}
\bigl(1+\mathcal X_p \bigr)^{-\alpha}
\mathcal D_v.
\end{cases}
\end{equation}
In the sequel, we derive the differential inequalities in \eqref{E-15} one by one. First, we derive the contraction property of the velocity diameter. 
\begin{lemma}
\label{L5.3}
Suppose that system parameters, communication weight function and initial datum satisfy 
\[ \kappa > 0, \quad \alpha \in [0, 1], \quad \phi:~\mbox{$\alpha$-moderate},  \quad p \in [1, \infty), \quad \mathcal X_p(0)<\infty, \quad\mathcal D_v(0)<\infty, \]
and let $\mu_t=(X(t,\cdot),V(t,\cdot))_{\#}\mu_0$ be a global Lagrangian weak
solution to \eqref{A-3}. Then, ${\mathcal D}_v$ satisfies $\eqref{E-15}_2$:
\begin{equation} \label{E-16}
D^+\mathcal D_v
\le
- \frac{\kappa}{2C_{\phi}}
\min\bigl\{6^{-\alpha},d_{\phi}\bigr\}
\bigl(1+\mathcal X_p \bigr)^{-\alpha}
\mathcal D_v, \quad \mbox{a.e.}~t > 0.
\end{equation}
\end{lemma}
\begin{proof}
At time $t$, we apply Lemma~\ref{L5.2} to the configuration
$X(t,\cdot)$. Then, we obtain
\begin{equation}\label{E-17}
\operatorname*{ess\,sup}_{z, {\widetilde z} \in {\mathbb R}^{2d}}
\Big |(\mathscr T_{X(t)}V(t,\cdot))(z)
-(\mathscr T_{X(t)}V(t,\cdot))(\widetilde z) \Big |
\le(1-\eta(t))\mathcal D_v(t),
\end{equation}
where $\eta$ is defined as follows:
\[
\eta(t)
:= \frac{1}{2C_{\phi}}
\min\bigl\{6^{-\alpha},d_{\phi}\bigr\}
(1+\mathcal X_p(t))^{-\alpha}.
\]
Note that the equations for velocity trajectory can be written as
\[
\dot V(t,z)+\kappa V(t,z)
=\kappa\mathscr T_{X(t)}V(t,\cdot)(z).
\]
For $0\le s<t$ and $z \in {\mathbb R}^{2d}$, we have
\begin{equation} \label{E-18}
V(t,z)
=e^{-\kappa(t-s)}V(s,z)
+\kappa\int_s^t
 e^{-\kappa(t-\tau)}
(\mathscr T_{X(\tau)}V(\tau,\cdot))(z)\,\di\tau.
\end{equation}
Now, we take pairwise differences together with the equation for $V(t,\tilde z)$, the essential supremum, and use
\eqref{E-17} to obtain
\begin{align}
\mathcal D_v(t)
\le
e^{-\kappa(t-s)}\mathcal D_v(s)+
\kappa\int_s^t e^{-\kappa(t-\tau)}
(1-\eta(\tau))\mathcal D_v(\tau)\,\di\tau.
\label{E-19}
\end{align}
The same characteristic identity for $X$ and Minkowski's inequality give locally
\[
|\mathcal X_p(t)-\mathcal X_p(s)|
\le
\mathcal D_v(0)|t-s|.
\]
Hence $\mathcal X_p$ is locally Lipschitz and $\eta$ is continuous. Moreover, we use 
\[
|\mathcal D_v(t)-\mathcal D_v(s)|
\le 2\kappa\mathcal D_v(0)|t-s|
\]
to see $\mathcal D_v$ is locally Lipschitz  continuous.\newline 

\noindent For $t>s$, we subtract $\mathcal D_v(s)$ from both sides of
\eqref{E-19} and divide by $t-s$ to obtain
\begin{align}
\begin{aligned} \label{E-20}
	\frac{\mathcal D_v(t)-\mathcal D_v(s)}{t-s}\le
	\frac{e^{-\kappa(t-s)}-1}{t-s}\,
	\mathcal D_v(s)+
	\frac{\kappa}{t-s}
	\int_s^t
	e^{-\kappa(t-\tau)}
	(1-\eta(\tau))\mathcal D_v(\tau)\,\di\tau.
\end{aligned}
\end{align}
Note that the coefficient in the right-hand side of  \eqref{E-20} becomes 
\begin{equation} \label{E-21}
\frac{e^{-\kappa(t-s)}-1}{t-s}
\longrightarrow -\kappa, \quad \mbox{as}~t\downarrow s.
\end{equation}
On the other hand, since both $\eta$ and $\mathcal D_v$ are continuous,
\[
e^{-\kappa(t-\tau)}
(1-\eta(\tau))\mathcal D_v(\tau)
\longrightarrow
(1-\eta(s))\mathcal D_v(s) \quad \mbox{uniformly for $\tau\in[s,t]$ as $t\downarrow s$}.
\]
Therefore, we have
\begin{equation} \label{E-22}
\frac{1}{t-s}
\int_s^t
e^{-\kappa(t-\tau)}
(1-\eta(\tau))\mathcal D_v(\tau)\,\di\tau
\longrightarrow
(1-\eta(s))\mathcal D_v(s).
\end{equation}
In \eqref{E-20}, we take the upper right Dini derivative using \eqref{E-21} and \eqref{E-22} to find the desired estimate:
\[
	D^+\mathcal D_v(s)
	:=
	\limsup_{t\downarrow s}
	\frac{\mathcal D_v(t)-\mathcal D_v(s)}{t-s} \le
	-\kappa\mathcal D_v(s)
	+
	\kappa(1-\eta(s))\mathcal D_v(s)=
	-\kappa\eta(s)\mathcal D_v(s).
\]
\end{proof}
\begin{lemma}  \label{L5.4}
Under the same setting as in Lemma \ref{L5.3}, the following assertions hold.
\begin{align}
\begin{aligned} \label{E-23}
& (i)~\mathcal X_p(t)
\le
\mathcal X_p(s)
+
\int_s^t\mathcal D_v(\tau)\,\di\tau, \quad  \mbox{for every $0\le s\le t$}. \\
& (ii)~\mathcal X_p~ \mbox{is locally Lipschitz and} \quad  D^+\mathcal X_p(t)\le\mathcal D_v(t)
\qquad\text{for all }t\ge0.
\end{aligned}
\end{align}
\end{lemma}
\begin{proof}
\noindent (i)~
For $0\le s\le t$ and $z,z_{\star}\in\mathbb R^{2d}$, we have
\[
X(t,z)-X(t,z_{\star})
=
X(s,z)-X(s,z_{\star})
+
\int_s^t
\bigl(V(\tau,z)-V(\tau,z_{\star})\bigr)\,\di\tau.
\]
Again, we take $L^p(\mu_0\otimes\mu_0)$ norm of the above relation and use Minkowski's
inequality to get the first relation in \eqref{E-23}:
\begin{align*}
\mathcal X_p(t)
&\le
\mathcal X_p(s)
+
\int_s^t
\left(
\iint
|V(\tau,z)-V(\tau,z_{\star})|^p
\,\di\mu_0(z)\di\mu_0(z_{\star})
\right)^{1/p}\di\tau\\
&=
\mathcal X_p(s)+\int_s^t\mathcal V_p(\tau)\,\di\tau\le
\mathcal X_p(s)+\int_s^t\mathcal D_v(\tau)\,\di\tau.
\end{align*}

\noindent (ii)~ Since
$\mathcal D_v(\tau)\le\mathcal D_v(0)$ by
Lemma~\ref{L5.3}, $\mathcal X_p$ is locally
Lipschitz, and we take the upper right Dini derivative in
$\eqref{E-23}_1$ to derive the second relation $\eqref{E-23}_2$.
\end{proof}

\subsection{Proof of Theorem~\ref{T3.1}}
\label{sec:5.3}
Now we are ready to provide the proof of our second main result.  Note that with the help of Lemma \ref{L5.3} and Lemma \ref{L5.4},  we can reduce the flocking problem to a closed set of  scalar-valued differential inequalities. Before we move onto the proof, we present  estimates for the SDDI as follows. 

Let $f,g:[0,\infty)\to[0,\infty)$ be locally Lipschitz functions satisfying the following SDDI:
\begin{equation} 
\begin{cases} \label{E-24}
\displaystyle D^+f(t)\le g(t),\quad \mbox{a.e.}~t > 0, \vspace{6pt}\\
\displaystyle D^+g(t)\le-A(1+f(t))^{-\alpha}g(t),
\end{cases}
\end{equation}
for some $A>0$ and $0\le\alpha\le1$, and we also define
\begin{equation}\label{E-25}
\Psi_{\alpha}(r)
:=
\int_0^r(1+s)^{-\alpha}\,\di s
=
\begin{cases}
\displaystyle
\frac{(1+r)^{1-\alpha}-1}{1-\alpha},
&0\le\alpha<1,\\[2mm]
\log(1+r),&\alpha=1.
\end{cases}
\end{equation}
Then, we can derive estimates for $f$ and $g$ in the following lemma.
\begin{lemma} \label{L5.5}
Let $f,g:[0,\infty)\to[0,\infty)$ be locally Lipschitz functions satisfying \eqref{E-24}. Then, the following assertions hold. 
\begin{enumerate}
\item
The SDDI \eqref{E-24} is weakly stable in the sense that 
\begin{equation}\label{E-26}
g(t)+A\Psi_{\alpha}(f(t))
\le
g(0)+A\Psi_{\alpha}(f(0)),
\quad t\ge0.
\end{equation}
\item
The asymptotic flocking occurs:
\begin{align}
\begin{aligned} \label{E-27}
& (i)~ \sup_{0 \leq t < \infty} f(t)\le f_{\infty} := \begin{cases}
\displaystyle
\left(
(1+f(0))^{1-\alpha}
+\frac{(1-\alpha)g(0)}{A}
\right)^{\!\frac1{1-\alpha}}-1,
&0\le\alpha<1,\\[3mm]
\displaystyle
(1+f(0))\exp\!\left(\frac{g(0)}{A}\right)-1,
&\alpha=1.
\end{cases} \\
&(ii)~ g(t)
\le
g(0) \exp \Big[ - A(1+f_{\infty})^{-\alpha} t ], \quad t > 0.
\end{aligned}
\end{align}
\end{enumerate}
\end{lemma}
\begin{proof}
Since the proof follows the standard argument in \cite{flocking4,Cucker,cucker2}, we omit it here.
\end{proof}
\vspace{0.5cm}

\noindent Now, we return to the proof of Theorem~\ref{T3.1}. First, we set 
\[
f(t):=\mathcal X_p(t),
\quad
g(t):=\mathcal D_v(t),
\quad
A=\frac{\kappa}{2C_{\phi}}
\min\bigl\{6^{-\alpha},d_{\phi}\bigr\}.
\]
Then, it follows from Lemma~\ref{L5.3} and Lemma~\ref{L5.4} that 
\[
D^+ f(t)\le g(t),
\quad
D^+g(t)
\le
-A(1+f(t))^{-\alpha} g(t).
\]
Then, we can apply Lemma~\ref{L5.5} to see that there exists a positive constant $ \mathcal X_{p,\infty}$ such that 
\begin{equation}\label{E-28}
\sup_{0 \leq t < \infty} \mathcal X_p(t)\le \mathcal X_{p,\infty}, \quad\text{for all }t\ge0.
\end{equation}
Thus
\[
\sup_{t\ge0}\mathcal X_p(t)
\le\mathcal X_{p,\infty}<\infty.
\]
Substituting \eqref{E-28} into
\eqref{E-16}, we obtain
\[
D^+\mathcal D_v(t)
\le
-\lambda\mathcal D_v(t),
\quad
\lambda
:=\frac{\kappa}{2C_{\phi}}
\min\bigl\{6^{-\alpha},d_{\phi}\bigr\}
(1+\mathcal X_{p,\infty})^{-\alpha}>0.
\]
This yields
\begin{equation}\label{E-29}
\mathcal D_v(t)
\le
\mathcal D_v(0)e^{-\lambda t},
\qquad t\ge0.
\end{equation}
Finally, we combine \eqref{E-28} and \eqref{E-29} to derive an exponential weak support flocking in the sense of
Definition~\ref{D1.1}. This completes the proof.
\begin{remark}
	Our results can also be extended to regime $\alpha>1$ for restricted initial data similar to \cite{flocking4}.
\end{remark}

\section{Weak moment flocking}\label{sec:6}
\setcounter{equation}{0}
In this section, we provide a proof of Theorem \ref{T3.2} on the weak moment flocking. In what follows, we fix $p\ge2$.  As shown in
Lemma~\ref{L6.1} below, Proposition~\ref{P3.1} provides the global Lagrangian weak solution
$\mu_t=(X(t,\cdot),V(t,\cdot))_{\#}\mu_0$. We also use the exact decomposition
\begin{align}
\begin{aligned} \label{F-1}
& V(t,z) =m_0+e^{-\kappa t}q_0(z)+U(t,z),
\quad q_0(z)=v-m_0, \\
& U\in\mathcal C_{\mathrm{loc}}
\bigl([0,\infty);L^\infty(\mu_0;\mathbb R^d)\bigr).
\end{aligned}
\end{align}
The proof follows six steps outlined in Theorem~\ref{T3.2}. The main
point is that the unbounded initial velocity tail is confined to the
explicit factor $e^{-\kappa t}q_0$, whereas the bounded remainder $U$
obeys a normalized alignment equation with an exponentially decaying
forcing. We first record the moment consequences of the assumptions in
Theorem~\ref{T3.2}, and then we combine the normalized estimate from
Section~\ref{sec:4} with the common-mass contraction from
Section~\ref{sec:5}, and finally close the resulting scalar system.

\subsection{Pairwise moments and the unbounded initial tail}\label{sec:6.1}
Note that the assumptions of Theorem~\ref{T3.2} are formulated in terms of
pairwise moments. In the next lemma, we show that they imply the
absolute second moments required by Proposition~\ref{P3.1} and
also quantify the size of $q_0$.

\begin{lemma} \label{L6.1}
Suppose that parameter and initial datum satisfy 
\[ p \geq 2, \quad  \mathcal X_p(0)<\infty, \quad
\mathcal V_p(0)<\infty, \]
and let $\mu_t=(X(t,\cdot),V(t,\cdot))_{\#}\mu_0$ be a global Lagrangian weak
solution to \eqref{A-3}. Then, the following assertions hold.
\begin{enumerate}
\item
The initial measure $\mu_0$ satisfies the integrability:
\[
\mu_0\in\mathcal P_p(\mathbb R^{2d}), \quad \mbox{in particular}, \quad  \mu_0\in\mathcal P_2(\mathbb R^{2d}).
\]
\item
$m_0$ is well defined,  $q_0\in L^p(\mu_0;\mathbb R^d)$ and 
\begin{equation} \label{F-2}
\|q_0\|_{L^2(\mu_0)}
=\frac{1}{\sqrt2}\mathcal V_2(0)
\le\frac{1}{\sqrt2}\mathcal V_p(0).
\end{equation}
\item
For every time at which $\mathcal X_p(t)<\infty$,
\begin{equation}\label{F-3}
\mathcal X_2(t)\le\mathcal X_p(t).
\end{equation}
\end{enumerate}
\end{lemma}
\begin{proof}
(1)~We set 
\[
F_x(z):=\int_{\mathbb R^{2d}}|x-x_{\star}|^p\,\mu_0(\di z_{\star}),
\quad
F_v(z):=\int_{\mathbb R^{2d}}|v-v_{\star}|^p\,\mu_0(\di z_{\star}).
\]
By Tonelli's theorem, one has 
\[
\int F_x\,\di\mu_0=\mathcal X_p(0)^p<\infty,
\quad
\int F_v\,\di\mu_0=\mathcal V_p(0)^p<\infty.
\]
Hence there exists a label $z_0=(x_0,v_0)$ for which both $F_x(z_0)$
and $F_v(z_0)$ are finite. Now, we use  $|a+b|^p\le2^{p-1}(|a|^p+|b|^p)$ to obtain
\[
\int|x|^p\,\di\mu_0
\le2^{p-1}\left(F_x(z_0)+|x_0|^p\right)<\infty, \quad 
\int|v|^p\,\di\mu_0
\le2^{p-1}\left(F_v(z_0)+|v_0|^p\right)<\infty.
\]
Thus 
\[ \mu_0\in\mathcal P_p(\mathbb R^{2d}). \]
(2) and (3):~Note that $m_0=\int v\,\di\mu_0$ is well defined and $q_0=v-m_0\in L^p(\mu_0)$. Since $\mu_0$ is a probability measure and has a finite second velocity
moment, the variance identity yields
\[
\int|v-m_0|^2\,\di\mu_0
=\frac12\iint|v-v_{\star}|^2\,
\mu_0(\di z)\mu_0(\di z_{\star})
=\frac12\mathcal V_2(0)^2.
\]
Because $p\ge2$, monotonicity of $L^r$ norms on a probability space gives
\[ \mathcal V_2(0)\le\mathcal V_p(0), \]
which verifies \eqref{F-2}. The same argument on the product
probability space gives \eqref{F-3}.
\end{proof}
Next, we combine Lemma~\ref{L4.2} with the preceding
pairwise estimate to justify \textbf{Step B} in the outlined proof of Theorem~\ref{T3.2}.
\begin{lemma}
\label{L6.2}
Suppose that the same setting in Theorem~\ref{T3.2} holds. Then, there exists a constant
$C_{\rm tail}>0$, depending only on $C_\phi$, $\alpha$, and
$\mathcal V_p(0)$, such that
\begin{equation}\label{F-4}
\|\mathscr T_{X(t)}q_0\|_{L^\infty(\mu_0)}
\le \frac{2^{\alpha+1}C_{\phi} \mathcal V_p(0)}{\sqrt2} \bigl(1+\mathcal X_p(t)\bigr)^\alpha=:C_{\rm tail}\bigl(1+\mathcal X_p(t)\bigr)^\alpha,
\end{equation}
for every $t\ge0$ for which $\mathcal X_p(t)<\infty$. 
\end{lemma}
\begin{proof}
It follows from Lemma~\ref{L4.2} that 
\begin{equation} \label{F-5}
\|\mathscr T_{X(t)}q_0\|_{L^\infty}
\le 2^{\alpha+1}C_{\phi} \|q_0\|_{L^2(\mu_0)}
\bigl(1+\mathcal X_2(t)\bigr)^\alpha.
\end{equation}
By Lemma~\ref{L6.1} and $p\ge2$, we have
\begin{equation} \label{F-6}
\|q_0\|_{L^2(\mu_0)}\le\frac1{\sqrt2}\mathcal V_p(0),
\quad
\mathcal X_2(t)\le\mathcal X_p(t).
\end{equation}
Finally, we combine \eqref{F-5} and \eqref{F-6} to get \eqref{F-4}.
\end{proof}

\subsection{Remainder oscillation and the scalar reduction}\label{sec:6.2}
Define the oscillation of the bounded remainder by
\begin{equation}\label{F-7}
\mathcal D_U(t)
:=\operatorname*{ess\,sup}_{z,z_{\star}}
|U(t,z)-U(t,z_{\star})|
\end{equation}
and an auxiliary spatial scale:
\begin{equation}\label{F-8}
\mathcal S_U(t)
:=1+\mathcal X_p(0)+\frac1\kappa\mathcal V_p(0)
+\int_0^t\mathcal D_U(s)\,\di s.
\end{equation}
In the next lemma, we explain why this is the natural spatial envelope in the
fully noncompact regime.
\begin{lemma} \label{L6.3}
Suppose that the same setting in Theorem~\ref{T3.2} holds, and let $\mu_t$ be the global Lagrangian weak solution furnished by Proposition~\ref{P3.1}. Then, the following assertions hold.
\begin{enumerate}
\item
$\mathcal D_U$ is finite on
every compact time interval and
\begin{equation}\label{F-9}
1+\mathcal X_p(t)\le\mathcal S_U(t),
\qquad t\ge0.
\end{equation}
\item
$\mathcal D_U$ is locally Lipschitz,
$\mathcal S_U\in C^1([0,\infty))$ and
\begin{equation}\label{F-10}
\mathcal S_U'(t)=\mathcal D_U(t),
\qquad
\mathcal D_U(0)=0.
\end{equation}
\end{enumerate}
\end{lemma}
\begin{proof}
\noindent (1)~
Since $U\in C_{\mathrm{loc}}([0,\infty);L^\infty(\mu_0))$ by
Proposition~\ref{P3.1}, $\mathcal D_U(t)<\infty$ on every
compact time interval. Moreover, it follows from \eqref{D-12} that for
$z,z_{\star}\in\mathbb R^{2d}$,
\begin{align*}
X(t,z)-X(t,z_{\star})
={}&x-x_{\star}
+\frac{1-e^{-\kappa t}}{\kappa}
\bigl(q_0(z)-q_0(z_{\star})\bigr)+\int_0^t
\bigl(U(s,z)-U(s,z_{\star})\bigr)\,\di s.
\end{align*}
Since $q_0(z)-q_0(z_{\star})=v-v_{\star}$, Minkowski's inequality in
$L^p(\mu_0\otimes\mu_0)$ yields the desired estimate \eqref{F-9}:
\begin{align}
\mathcal X_p(t)
&\le\mathcal X_p(0)
+\frac{1-e^{-\kappa t}}{\kappa}\mathcal V_p(0)
+\int_0^t
\left(\iint|U(s,z)-U(s,z_{\star})|^p\,
\di\mu_0(z)\di\mu_0(z_{\star})\right)^{1/p}\di s
\nonumber\\
&\le\mathcal X_p(0)+\frac1\kappa\mathcal V_p(0)
+\int_0^t\mathcal D_U(s)\,\di s.
\label{F-11}
\end{align}
\noindent (2)~It remains to justify the time regularity. Indeed, we differentiate two Volterra terms in $\eqref{D-4}_3$ to get 
\begin{align}
\begin{aligned} \label{F-12}
& \frac{\di}{\di t}
	\left[
	\kappa\int_0^t e^{-\kappa(t-s)}
	\mathscr T_{X(s)}U(s,\cdot)(z)\,\di s
	\right] \\
& \hspace{1cm} = \kappa\mathscr T_{X(t)}U(t,\cdot)(z)  -\kappa^2\int_0^t e^{-\kappa(t-s)}
	\mathscr T_{X(s)}U(s,\cdot)(z)\,\di s, \\
& \frac{\di}{\di t}
	\left[
	\kappa\int_0^t e^{-\kappa(t-s)}e^{-\kappa s}
	(\mathscr T_{X(s)}q_0)(z)\,\di s
	\right] \\
& \hspace{1cm} =
	\kappa e^{-\kappa t}\mathscr T_{X(t)}q_0(z) 
	-\kappa^2\int_0^t e^{-\kappa(t-s)}e^{-\kappa s}
	(\mathscr T_{X(s)}q_0)(z)\,\di s.
\end{aligned}
\end{align}
We add two identities  in \eqref{F-12} and use $\eqref{D-4}_3$, the integral terms combine into
$-\kappa U(t,z)$, and hence we have
\begin{equation}\label{F-13}
\dot U(t,z)
=-\kappa U(t,z)
+\kappa\mathscr T_{X(t)}U(t,\cdot)(z)
+\kappa e^{-\kappa t}\mathscr T_{X(t)}q_0(z), \quad \mbox{a.e.,}~t > 0.
\end{equation}
On each interval $[0,T]$,
Proposition~\ref{P3.1}, the bound
\eqref{D-5}, and
Lemma~\ref{L6.2} imply
\[
\operatorname*{ess\ \sup}_{0\le t\le T}\|\dot U(t)\|_{L^\infty(\mu_0)}<\infty.
\]
Hence $U$ is Lipschitz from $[0,T]$ into $L^\infty(\mu_0)$ and
\[
|\mathcal D_U(t)-\mathcal D_U(s)|
\le2\|U(t)-U(s)\|_{L^\infty(\mu_0)}.
\]
Thus, $\mathcal D_U$ is locally Lipschitz. Therefore
\eqref{F-8} implies $\mathcal S_U\in C^1$ and
$\mathcal S_U'=\mathcal D_U$. Finally, $U(0)=0$ gives
$\mathcal D_U(0)=0$.
\end{proof}
In the next lemma, we estimate the principal dynamical inequality for the
fully noncompact problem. It combines the common-mass estimate from
Section~\ref{sec:5} with the normalized tail bound above. We set 
\begin{equation}\label{F-14}
a_{\rm align}:=\frac{\kappa}{2C_{\phi}}
\min\bigl\{6^{-\alpha},d_{\phi}\bigr\},
\qquad
b_{\rm tail}:=\frac{2^{\alpha+2}\kappa C_{\phi}}{\sqrt2}\mathcal V_p(0).
\end{equation}
\begin{lemma}
\label{L6.4}
Suppose that the same setting in Theorem~\ref{T3.2} holds, and let $\mu_t$ be the global Lagrangian weak solution furnished by Proposition~\ref{P3.1}. Then $\mathcal D_U$ satisfies 
\begin{equation}\label{F-15}
D^+\mathcal D_U
\le
-a_{\rm align}\mathcal S_U^{-\alpha}\mathcal D_U
+b_{\rm tail} e^{-\kappa t}\mathcal S_U^\alpha,
\quad \mbox{a.e.}~t >0.
\end{equation}
\end{lemma}
\begin{proof}
We first fix $0\le s<t$ and we use the variation-of-constants formula to 
\eqref{F-13} to get 
\begin{align*}
U(t,z)
={}&e^{-\kappa(t-s)}U(s,z)
+\kappa\int_s^t e^{-\kappa(t-\tau)}
\mathscr T_{X(\tau)}U(\tau,\cdot)(z)\,\di\tau\\
&+\kappa\int_s^t e^{-\kappa(t-\tau)}e^{-\kappa\tau}
\mathscr T_{X(\tau)}q_0(z)\,\di\tau.
\end{align*}
At each time $\tau$, Lemma~\ref{L5.2} applied to the
configuration $X(\tau,\cdot)$ gives
\[
\operatorname*{ess\,diam}
\bigl(\mathscr T_{X(\tau)}U(\tau,\cdot)\bigr)
\le
\bigl(1-\eta(\tau)\bigr)\mathcal D_U(\tau),
\]
where
\[
\eta(\tau)
:= \frac{1}{2C_{\phi}}
\min\bigl\{6^{-\alpha},d_{\phi}\bigr\} \bigl(1+\mathcal X_p(\tau)\bigr)^{-\alpha}.
\]
For the forcing term, we use Lemma~\ref{L6.2} to see
\[
\operatorname*{ess\,diam}
\bigl(\mathscr T_{X(\tau)}q_0\bigr)
\le2\|\mathscr T_{X(\tau)}q_0\|_{L^\infty}
\le2C_{\rm tail}\bigl(1+\mathcal X_p(\tau)\bigr)^\alpha.
\] Taking pairwise differences, then the
essential supremum, yields
\begin{align}
\mathcal D_U(t)
&\le e^{-\kappa(t-s)}\mathcal D_U(s)
+\kappa\int_s^t e^{-\kappa(t-\tau)}
\bigl(1-\eta(\tau)\bigr)\mathcal D_U(\tau)\,\di\tau
\nonumber\\
&\quad
+2\kappa C_{\rm tail}\int_s^t e^{-\kappa(t-\tau)}e^{-\kappa\tau}
\bigl(1+\mathcal X_p(\tau)\bigr)^\alpha\,\di\tau.
\label{F-16}
\end{align}
By Lemma~\ref{L6.3}, we have 
\[ 1+\mathcal X_p(\tau)\le\mathcal S_U(\tau). \] Hence, we get
\[
\eta(\tau)\ge \frac{1}{2C_{\phi}}
\min\bigl\{6^{-\alpha},d_{\phi}\bigr\} \mathcal S_U(\tau)^{-\alpha},
\qquad
\bigl(1+\mathcal X_p(\tau)\bigr)^\alpha
\le\mathcal S_U(\tau)^\alpha.
\]
Since $\mathcal D_U$ and $\mathcal S_U$ are continuous, we divide
\eqref{F-16} by $t-s$ and let $t\downarrow s$ to obtain the desired estimate:
\[
D^+\mathcal D_U(s)
\le
-\frac{\kappa}{2C_{\phi}}
\min\bigl\{6^{-\alpha},d_{\phi}\bigr\} \mathcal S_U(s)^{-\alpha}\mathcal D_U(s)
+2\kappa C_{\rm tail} e^{-\kappa s}\mathcal S_U(s)^\alpha.
\]
\end{proof}
Now, we study the estimate outlined in \textbf{Step E} of
Theorem~\ref{T3.2}. The proof separates the subcritical range
$0\le\alpha<1$ from the critical endpoint $\alpha=1$.

\begin{lemma} \label{L6.5}
Let ${\mathcal S}, {\mathcal D}:[0,\infty)\to[0,\infty)$ be locally Lipschitz, and they satisfy 
\begin{equation}\label{F-17}
\begin{cases}
\displaystyle D^+\mathcal D(t)
\le-a \mathcal S(t)^{-\alpha}\mathcal D(t)
+b e^{-\kappa t}\mathcal S(t)^\alpha, \quad \mbox{a.e.}~t > 0, \vspace{6pt}\\
\displaystyle  \mathcal S^{\prime}(t) =\mathcal D(t),\vspace{6pt}\\
\displaystyle  (\mathcal D, {\mathcal S}) \Big|_{t = 0} = (0, {\mathcal S}_0), \quad {\mathcal S}_0 \geq 1,
\end{cases}
\end{equation}
where parameters satisfy the following conditions:
\[ a >0,\quad \kappa>0, \quad b\ge0, \quad 0\le\alpha\le1. \]
Then there exist constants ${\mathcal S}_\infty$, $C_D$, and $\lambda$, depending only on
$a,b,\kappa,\alpha$, and ${\mathcal S}(0)$ such that
\begin{equation}\label{F-18}
\sup_{t\ge0} {\mathcal S}(t)\le {\mathcal S}_\infty,
\quad
{\mathcal D}(t)\le C_D e^{-\lambda t},
\quad t\ge0.
\end{equation}
\end{lemma}
\begin{proof}
We leave the detailed proof  in Appendix \ref{app-A}.
\end{proof}

\subsection{Proof of Theorem~\ref{T3.2}}\label{sec:6.3}
We are now ready to prove the weak moment flocking. By Lemma~\ref{L6.1}, the assumptions
in Theorem \ref{T3.2} imply $\mu_0\in\mathcal P_2(\mathbb R^{2d})$.
Hence Proposition~\ref{P3.1} provides the global Lagrangian weak 
solution and the decomposition \eqref{F-1} with
\[U\in C_{\mathrm{loc}}([0,\infty);L^\infty(\mu_0)).\]
For simplicity, we set
\[
X_{p,0}:=\mathcal X_p(0),
\quad
V_{p,0}:=\mathcal V_p(0),
\quad
\mathcal  S_0:=1+X_{p,0}+\frac{V_{p,0}}{\kappa}.
\]
By Lemma~\ref{L6.3} and Lemma ~\ref{L6.4},
\[
\mathcal S(t):=\mathcal S_U(t),
\quad
\mathcal{D}(t):=\mathcal D_U(t)
\]
satisfy Lemma~\ref{L6.5} with $a=a_{\rm align}$ and $b=b_{\rm tail}$, where
\begin{align}\label{NewF-19}
a_{\rm align}
:=
\frac{\kappa}{2C_\phi}
\min\bigl\{6^{-\alpha},d_\phi\bigr\},
\quad
b_{\rm tail}
:=
2^{\alpha+\frac32}\kappa C_\phi V_{p,0},
\quad
\mathcal  S(0)=\mathcal  S_0.
\end{align}
Define
\[
\mathcal S_\infty
:=
\mathfrak S_\alpha(a_{\rm align},b_{\rm tail},\kappa,\mathcal  S_0),
\]
where $\mathfrak S_\alpha$ is the explicit function introduced in the proof of
Lemma~\ref{L6.5} (see $\eqref{A.3}_2$), and we set
\[
\lambda_U
:=
\frac12
\min\left\{
a_{\rm align}\mathcal  S_\infty^{-\alpha},\kappa
\right\},
\quad
C_U
:=
\frac{2b_{\rm tail}}{a_{\rm align}}\mathcal  S_\infty^{2\alpha}.
\]
Then Lemma~\ref{L6.5} gives
\begin{equation}\label{F-19}
\sup_{t\ge0}\mathcal S_U(t)\le \mathcal  S_\infty,
\quad
\mathcal D_U(t)\le C_Ue^{-\lambda_Ut},
\qquad t\ge0.
\end{equation}
In particular, it follows from \eqref{F-9} that
\begin{equation}\label{F-20}
\sup_{t\ge0}\mathcal X_p(t)
\le \mathcal S_\infty-1.
\end{equation}
Finally, we take pairwise differences in \eqref{F-1} to find 
\[
V(t,z)-V(t,z_{\star})
=e^{-\kappa t}(v-v_{\star})
+U(t,z)-U(t,z_{\star}).
\]
Minkowski's inequality in $L^p(\mu_0\otimes\mu_0)$ therefore yields
\begin{equation}\label{F-21}
\mathcal V_p(t)
\le e^{-\kappa t}V_{p,0}+\mathcal D_U(t).
\end{equation}
Since $\lambda_U\le\kappa/2<\kappa$, \eqref{F-19} and \eqref{F-21}
imply
\[
\mathcal V_p(t)
\le
\bigl(V_{p,0}+C_U\bigr)e^{-\lambda_Ut},
\qquad t\ge0.
\]
Thus the constants in Theorem~\ref{T3.2} may be chosen explicitly as
\[
\lambda:=\lambda_U
=
\frac12
\min\left\{
a_{\rm align}\mathcal S_\infty^{-\alpha},\kappa
\right\},
\qquad
C:=1+V_{p,0}+C_U
=
1+V_{p,0}+\frac{2b_{\rm tail}}{a_{\rm align}}\mathcal  S_\infty^{2\alpha},
\]
with
\[
\mathcal S_\infty
=
\mathfrak S_\alpha(a_{\rm align},b_{\rm tail},\kappa,\mathcal S_0),
\qquad
\mathcal  S_0=1+X_{p,0}+\frac{V_{p,0}}{\kappa},
\]
and $a_{\rm align},b_{\rm tail}$ as above in \eqref{NewF-19}. Hence $\mathcal S_\infty$, $C$, and $\lambda$ depend
only on
\[
\kappa,\quad
\alpha,\quad
C_\phi,\quad
d_\phi,\quad
\mathcal X_p(0),\quad
\mathcal V_p(0).
\]
This proves exponential weak moment flocking in the sense of
Definition~\ref{D1.1}.

\vspace{.3cm}

\section{Conclusion}\label{sec:7}
In this paper, we have studied the emergent dynamics of the phase-spatially extended KMT model without compactness assumptions on the spatial-velocity support. We first established a direct global Lagrangian well-posedness theory in the fully noncompact finite-moment regime by separating the unbounded initial velocity tail from a bounded interaction-generated remainder. For compact velocity support, a time-varying effective-region argument yields uniform spatial cohesion together with exponential contraction of the velocity-support diameter, and hence exponential weak support flocking. In the fully noncompact position--velocity regime, support-level alignment is in general unavailable. Nevertheless, the same normalized interaction mechanism, combined with the remainder decomposition, yields exponential decay of the pairwise velocity moment and therefore exponential weak moment flocking. These results show that the normalized MT interaction remains effective even in the absence of momentum conservation and compact phase-space support. Several directions remain open for future investigation. It would be interesting to extend the present noncompact framework to non-symmetric kinetic alignment models subject to stochastic or environmental forcing, where unbounded velocity tails arise naturally.  Moreover, the applicability of the effective-region and remainder-decomposition methods to other non-symmetric alignment models deserves further investigation. We leave these interesting problems for future study.

\vspace{.3cm}
\section*{Conflict of interest statement}
The authors declare no conflicts of interest.

\section*{Data availability statement}
The data supporting the findings of this study are available from the corresponding author upon reasonable request.

\vspace{1.5cm}

\appendix 

\section{Proof of Lemma \ref{L6.5}}\label{app-A}
In this appendix, we prove 
Lemma~\ref{L6.5}. For $\xi>0$ and $r\ge0$, we define
\[
\mathcal J_{\xi,r}
:=
\int_0^\infty e^{-\xi t}(1+t)^r\,\di t
=
e^\xi\xi^{-r-1}\Gamma(r+1,\xi),
\]
where $\Gamma(\cdot,\cdot)$ denotes the upper incomplete Gamma function
\[
\Gamma(s,x):=\int_x^\infty t^{s-1}e^{-t}\,\mathrm{d}t,
\qquad x>0,~ s>0.
\]
Moreover, we set
\[
{\mathcal D}_{\rm lin}:=
\frac{b {\mathcal S}_0}{\kappa}
\exp\!\left(\frac{b}{\kappa^2}\right),
\quad
{\mathcal S}_{\rm lin}:= {\mathcal S}_0+ {\mathcal D}_{\rm lin}.
\]
If $0\le\alpha<1$, we define
\begin{align}\label{A.1}
\gamma:=\alpha,
\quad
{\mathcal G}_\gamma:= {\mathcal S}_{\rm lin}.
\end{align}
If $\alpha=1$, we define
\begin{align}\label{A.2}
\delta_{\rm c}:=\frac{a}{{\mathcal S}_{\rm lin}},
\quad
{\mathcal K}_{\rm c}:=b {\mathcal S}_{\rm lin}\,\mathcal J_{\kappa,1+\delta_{\rm c}},
\quad
\gamma:=1-\frac12\min\{\delta_{\rm c},1\},
\quad
{\mathcal G}_\gamma:= {\mathcal S}_0+\frac{{\mathcal K}_{\rm c}}{\gamma}.
\end{align}
In both cases $0\le\gamma<1$. We set 
\begin{align}\label{A.3}
	\begin{aligned}
		& \vartheta_\gamma:=1-\gamma,
		\quad
		a_\gamma:=a {\mathcal G}_\gamma^{-\alpha},
		\quad
		b_\gamma:=b {\mathcal G}_\gamma^\alpha, \\
		& \mathfrak S_\alpha(a,b,\kappa, {\mathcal S}_0)
		:=
		{\mathcal S}_0+2 {\mathcal D}_{\rm lin}
		+b_\gamma\mathcal J_{\kappa,\gamma}
		\Gamma\!\left(1+\frac1{\vartheta_\gamma}\right)
		\left(\frac{3}{a_\gamma}\right)^{1/\vartheta_\gamma}
		+\frac{b_\gamma}{\kappa}\mathcal J_{\kappa/2,\gamma}.
	\end{aligned}
\end{align}
Here, $\Gamma(\cdot)$ denotes the standard Gamma function 
\[
\Gamma(z)=\int_{0}^{\infty} t^{z-1}e^{-t}\di t.
\] Then one may take
\begin{align}\label{A.4}
{\mathcal S}_\infty:=\mathfrak S_\alpha(a,b,\kappa, {\mathcal S}_0),
\qquad
\lambda:=
\frac12\min\left\{a {\mathcal S}_\infty^{-\alpha},\kappa\right\},
\quad
C_D:=
\frac{2b}{a} {\mathcal S}_\infty^{2\alpha}
\end{align}
in Lemma \ref{L6.5}. We split the proof into four steps. 
\vspace{0.2cm}

\noindent
$\bullet$ {\bf Step A} (A priori bounds):~Since $\mathcal {S}\ge1$ and $0\le\alpha\le1$, we have 
\[ \mathcal S^\alpha\le \mathcal S. \]
Then, we use this and drop the non-positive damping term in $\eqref{F-17}_1$ to see that 
\[
D^+\mathcal D(t) \le-a \mathcal S(t)^{-\alpha}\mathcal D(t)
+b e^{-\kappa t}\mathcal S(t)^\alpha \leq b e^{-\kappa t}\mathcal S(t)^\alpha \leq b e^{-\kappa t}\mathcal S(t),
\]
i.e.,  at almost
every differentiability point, we have
\[
\begin{cases}
\displaystyle  {\mathcal D}'(t)\le be^{-\kappa t} {\mathcal S}(t), \quad \mbox{a.e.}~t > 0, \vspace{6pt}\\
\displaystyle {\mathcal D}(0)=0.
\end{cases}
\]
We integrate this from $0$ to $t$ to get 
\begin{equation}\label{A.5}
{\mathcal D}(t)\le b{\mathcal H}(t),
\quad
{\mathcal H}(t):=\int_0^t e^{-\kappa s} {\mathcal S}(s)\,\di s.
\end{equation}
Since
\[
{\mathcal S}(t)= {\mathcal S}_0+\int_0^t {\mathcal D}(r)\,\di r,
\]
Fubini's theorem yields
\[
\begin{aligned}
{\mathcal H}(t)
&=
{\mathcal S}_0\int_0^t e^{-\kappa s}\,\di s
+\int_0^t e^{-\kappa s}\int_0^s {\mathcal D}(r)\,\di r\,\di s
\\
&\le
\frac{{\mathcal S}_0}{\kappa}
+\frac1\kappa\int_0^t e^{-\kappa r} {\mathcal D}(r)\,\di r\le
\frac{ {\mathcal S}_0}{\kappa}
+\frac b\kappa\int_0^t e^{-\kappa r} {\mathcal H}(r)\,\di r.
\end{aligned}
\]
Grönwall's inequality gives
\begin{equation}\label{A.6}
{\mathcal H}(t)
\le
\frac{ {\mathcal S}_0}{\kappa}
\exp\!\left(\frac{b}{\kappa^2}\right),
\qquad t\ge0.
\end{equation}
Then, \eqref{A.5} and \eqref{A.6} imply
\begin{equation}\label{A.7}
{\mathcal D}(t)\le {\mathcal D}_{\rm lin}
=
\frac{b {\mathcal S}_0}{\kappa}
\exp\!\left(\frac{b}{\kappa^2}\right),
\quad
{\mathcal S}(t)\le {\mathcal S}_0+ {\mathcal D}_{\rm lin}t\le  {\mathcal S}_{\rm lin}(1+t),
\end{equation}
where ${\mathcal S}_{\rm lin}:= {\mathcal S}_0+ {\mathcal D}_{\rm lin}$. \newline

\noindent $\bullet$ {\bf Step B} (Construction of an effective subcritical exponent):~For $\alpha \in [0, 1),$ we set  
\[
\gamma:=\alpha, \quad
{\mathcal G}_\gamma:= {\mathcal S}_{\rm lin}.
\]
Then the linear bound \eqref{A.7} directly gives
\[
{\mathcal S}(t)^{-\alpha}
\ge {\mathcal G}_\gamma^{-\alpha}(1+t)^{-\gamma},
\quad
{\mathcal S}(t)^\alpha
\le {\mathcal G}_\gamma^\alpha(1+t)^\gamma.
\]
For $\alpha=1$, we use \eqref{A.7} and $\eqref{F-17}_1$ to get 
\[
D^+ {\mathcal D}(t)
\le
-\delta_{\rm c}(1+t)^{-1} {\mathcal D}(t)
+b {\mathcal S}_{\rm lin}e^{-\kappa t}(1+t),
\quad
\delta_{\rm c}:=\frac{a}{{\mathcal S}_{\rm lin}}.
\]
Then, we multiply by $(1+t)^{\delta_{\rm c}}$ and integrate the resulting relation to get 
\[
{\mathcal D}(t)
\le
b {\mathcal S}_{\rm lin}(1+t)^{-\delta_{\rm c}}
\int_0^t e^{-\kappa s}(1+s)^{1+\delta_{\rm c}}\,\di s
\le
{\mathcal K}_{\rm c}(1+t)^{-\delta_{\rm c}},
\]
where
\[
{\mathcal K}_{\rm c}:=b{\mathcal S}_{\rm lin}\,\mathcal J_{\kappa,1+\delta_{\rm c}}.
\]
Next, we choose
\[
\gamma:=1-\frac12\min\{\delta_{\rm c},1\}.
\]
Then, we have 
\[
1/2\le\gamma<1 \quad \mbox{and}  \quad 1-\gamma=\frac12\min\{\delta_{\rm c},1\}\le\delta_{\rm c}.
\]
Hence,
\[
(1+s)^{-\delta_{\rm c}}\le(1+s)^{\gamma-1}.
\]
Therefore, we have
\[
{\mathcal S}(t)
=
{\mathcal S}_0+\int_0^t {\mathcal D}(s)\,\di s\le
{\mathcal S}_0+ {\mathcal K}_{\rm c}\int_0^t(1+s)^{\gamma-1}\,\di s
\le
\left({\mathcal S}_0+\frac{ {\mathcal K}_{\rm c}}{\gamma}\right)(1+t)^\gamma
=
{\mathcal G}_\gamma(1+t)^\gamma.
\]
Since $\alpha=1$, this implies
\[
{\mathcal S}(t)^{-1}
\ge {\mathcal G}_\gamma^{-1}(1+t)^{-\gamma},
\quad
{\mathcal S}(t)\le {\mathcal G}_\gamma(1+t)^\gamma.
\]
In both cases, with the explicit $\gamma$ and ${\mathcal G}_\gamma$ stated above in \eqref{A.1} and \eqref{A.2}, we have
\begin{equation}\label{A.8}
{\mathcal S}(t)^{-\alpha}
\ge {\mathcal G}_\gamma^{-\alpha}(1+t)^{-\gamma},
\quad
{\mathcal S}(t)^\alpha
\le{\mathcal G}_\gamma^\alpha(1+t)^\gamma,
\quad
0\le\gamma<1.
\end{equation}
Therefore $\eqref{F-17}_1$ yields
\begin{equation}\label{A.9}
D^+ {\mathcal D}(t)
\le
-a_\gamma(1+t)^{-\gamma} {\mathcal D}(t)
+b_\gamma e^{-\kappa t}(1+t)^\gamma,
\end{equation}
where $a_\gamma$ and $b_\gamma$ are given as follows. 
\[
a_\gamma:=a{\mathcal G}_\gamma^{-\alpha},
\quad
b_\gamma:=b{\mathcal G}_\gamma^\alpha.
\]
\noindent $\bullet$ {\bf Step C} (Explicit integrability and uniform boundedness of $\mathcal S$):~
By the comparison for the scalar differential inequality \eqref{A.9},
\[
{\mathcal D}(t)
\le
b_\gamma\int_0^t e^{-\kappa s}(1+s)^\gamma
\exp\!\left(
-a_\gamma\int_s^t(1+\tau)^{-\gamma}\,\di\tau
\right)\di s.
\]
We set 
\[
\vartheta_\gamma:=1-\gamma>0.
\]
\noindent $\diamond$~{\bf Case 1}:~For $t\ge2$, we split the integral at $t/2$:
\[ \mbox{Either}~~0 \leq s \leq \frac{t}{2}, \quad  \mbox{or} \quad  \frac{t}{2} \leq s \leq t. \]
If $0\le s\le t/2$, then
\[
\begin{aligned}
\int_s^t(1+\tau)^{-\gamma}\,\di\tau
\ge
\int_{t/2}^t(1+\tau)^{-\gamma}\,\di\tau
\ge
\frac{t}{2}(1+t)^{-\gamma}
\ge
\frac13(1+t)^{\vartheta_\gamma}.
\end{aligned}
\]
Hence the contribution of $[0,t/2]$ is bounded by
\[
b_\gamma\mathcal J_{\kappa,\gamma}
\exp\!\left(
-\frac{a_\gamma}{3}(1+t)^{\vartheta_\gamma}
\right).
\]
If $t/2\le s\le t$, the damping exponential is at most one,
and thus
\[
\begin{aligned}
b_\gamma\int_{t/2}^t e^{-\kappa s}(1+s)^\gamma\,\di s
\le
b_\gamma(1+t)^\gamma
\int_{t/2}^\infty e^{-\kappa s}\,\di s
=
\frac{b_\gamma}{\kappa}
(1+t)^\gamma e^{-\kappa t/2}.
\end{aligned}
\]
Consequently, we have
\begin{equation}\label{A.10}
{\mathcal D}(t)
\le
b_\gamma\mathcal J_{\kappa,\gamma}
e^{-\frac{a_\gamma}{3}(1+t)^{\vartheta_\gamma}}
+
\frac{b_\gamma}{\kappa}
(1+t)^\gamma e^{-\kappa t/2}.
\end{equation}
\noindent $\diamond$~{\bf Case 2}:~For $0\le t\le2$, the relation \eqref{A.7} gives 
\[ {\mathcal D}(t)\le {\mathcal D}_{\rm lin}. \]
Therefore, we have
\[
\begin{aligned}
\int_0^\infty {\mathcal D}(t)\,\di t
\le
2 {\mathcal D}_{\rm lin}
+
b_\gamma\mathcal J_{\kappa,\gamma}
\int_0^\infty
e^{-\frac{a_\gamma}{3}(1+t)^{\vartheta_\gamma}}\,\di t
+
\frac{b_\gamma}{\kappa}
\mathcal J_{\kappa/2,\gamma}.
\end{aligned}
\]
Since $(1+t)^{\vartheta_\gamma}\ge t^{\vartheta_\gamma}$,
\[
\begin{aligned}
\int_0^\infty
e^{-\frac{a_\gamma}{3}(1+t)^{\vartheta_\gamma}}\,\di t
\le
\int_0^\infty
e^{-\frac{a_\gamma}{3}t^{\vartheta_\gamma}}\,\di t
=
\Gamma\!\left(1+\frac1{\vartheta_\gamma}\right)
\left(\frac{3}{a_\gamma}\right)^{1/\vartheta_\gamma}.
\end{aligned}
\]
It follows that
\[
\int_0^\infty {\mathcal D}(t)\,\di t
\le
{\mathcal S}_\infty- {\mathcal S}_0,
\quad 
\mbox{where}
\quad 
{\mathcal S}_\infty
=
\mathfrak S_\alpha(a,b,\kappa,{\mathcal S}_0)
\]
is exactly the explicit constant stated in $\eqref{A.3}_2$. Since
\[ {\mathcal S}(t)= {\mathcal S}_0+\int_0^t {\mathcal D}(s)\,\di s, \]
we conclude that
\[
\sup_{t\ge0} {\mathcal S}(t)\le {\mathcal S}_\infty.
\]
\noindent $\bullet$ {\bf Step D} (Explicit exponential rate):~
We again return to $\eqref{F-17}_1$ to obtain
\[
D^+ {\mathcal D}(t)
\le
-a {\mathcal S}_\infty^{-\alpha}{\mathcal D}(t)+b {\mathcal S}_\infty^\alpha e^{-\kappa t}.
\]
Since ${\mathcal D}(0)=0$, comparison principle yields
\[
{\mathcal D}(t)
\le b {\mathcal S}_\infty^\alpha
\int_0^t
e^{-a {\mathcal S}_\infty^{-\alpha}(t-s)}e^{-\kappa s}\,\di s.
\]
We choose
\[
\lambda
:=
\frac12\min\{a {\mathcal S}_\infty^{-\alpha},\kappa\}.
\]
Then $\lambda\le a {\mathcal S}_\infty^{-\alpha}/2$ and
\[
\begin{aligned}
{\mathcal D}(t)
\le
b {\mathcal S}_\infty^\alpha e^{-\lambda t}
\int_0^t
e^{-(a {\mathcal S}_\infty^{-\alpha}-\lambda)(t-s)}
e^{-(\kappa-\lambda)s}\,\di s
\le
\frac{b {\mathcal S}_\infty^\alpha}{a {\mathcal S}_\infty^{-\alpha}-\lambda}e^{-\lambda t}
\le
\frac{2b {\mathcal S}_\infty^\alpha}{a {\mathcal S}_\infty^{-\alpha}}e^{-\lambda t}
=
\frac{2b}{a}{\mathcal S}_\infty^{2\alpha}e^{-\lambda t}.
\end{aligned}
\]
This verifies \eqref{F-18} with the explicit constants claimed in \eqref{A.4}.

\end{document}